\documentclass[a4paper,11pt]{article}
\usepackage[left=2.5cm,right=2.5cm,top=3.5cm,bottom=3.5cm]{geometry}

\usepackage{amsmath,amssymb,amsthm,bm,mathtools,xspace,booktabs,url}
\usepackage{graphicx,etoolbox,xcolor,color,dsfont,hyperref,bbold,tikz} 
\usepackage{mathrsfs}
\usepackage[shortlabels]{enumitem}
\newcommand{\typered}[1]{{\textcolor{red}{#1}}}
\newcommand{\typeblue}[1]{{\textcolor{blue}{#1}}}
\newcommand{\orbit}[0]{\boldsymbol{\mathcal{A}}}
\DeclareMathAlphabet\mathbfcal{OMS}{cmsy}{b}{n}
\makeatletter
\global\let\tikz@ensure@dollar@catcode=\relax
\makeatother
\setlist{
  listparindent=\parindent,
  parsep=0pt,
}
\usepackage{amssymb}

\usepackage{cleveref}

\AtBeginEnvironment{thebibliography}{
  \small
  \setlength\itemsep{0.75em plus 0.5em minus 0.75em}%
}

\usepackage{caption}
\usepackage[raggedright]{titlesec}

\numberwithin{equation}{section}
\usepackage{amsfonts}

\theoremstyle{plain} 
\newtheorem{theorem}{Theorem}[section]
\newtheorem{Lemma}[theorem]{Lemma}
\newtheorem{Proposition}[theorem]{Proposition}
\newtheorem{Corollary}[theorem]{Corollary}
\newtheorem{Condition}[theorem]{Condition}
\newtheorem{definition}[theorem]{Definition}
\newtheorem{remark}[theorem]{Remark}
\newtheorem{assumption}[theorem]{Assumption}

\theoremstyle{definition} 
\newtheorem{Example}[theorem]{Example}

\usepackage{authblk}

\makeatletter
\newcommand\CorrespondingAuthor[1]{
  \begingroup
  \def\@makefnmark{}
  \footnotetext{Corresponding author: #1}
  \endgroup
}
\makeatother

\makeatletter
\renewenvironment{abstract}{%
  \small%
  \providecommand\keywords{%
    \par\medskip\noindent\textit{Keywords:}\xspace}%
  \begin{center}%
    \bfseries \abstractname\vspace{-.5em}\vspace{\z@}%
  \end{center}%
  \quote%
}
{
\endquote}
\makeatother

\usepackage[above,below,section]{placeins}

\usepackage[all]{onlyamsmath}

\usepackage{multirow}

\usepackage{xcolor}
\usepackage{color}

\usepackage{mathrsfs}
\DeclareMathAlphabet{\mathpzc}{T1}{pzc}{m}{it}

\newcommand{\Indic}[1]{\mathds{1}{\left\{#1\right\}}}

\newcommand{\dd}{\mathrm{d}}

\def\*{\discretionary{}{\hbox{\ensuremath\cdot}\thinspace}{}}

\usepackage{tabularx,ragged2e,booktabs}

\usepackage{pdfpages}

\definecolor{darkmagenta}{rgb}{0.5,0,0.5}
\definecolor{darkgreen}{rgb}{0.2,0.3,0}
\definecolor{darkblue}{rgb}{0,0,0.6}
\definecolor{darkred}{rgb}{0.8,0,0}
\definecolor{mellow}{rgb}{.847, 0.72, 0.525}

\newcommand{\magenta}[1]{\textcolor{darkmagenta}{#1}}

\begin{document}

\title{RAP-modulated Fluid Processes: \\
First Passages and the Stationary Distribution}
\author[1]{Nigel G. Bean}
\author[1]{Giang T. Nguyen}
\author[2]{Bo F. Nielsen}
\author[1]{Oscar Peralta}
\affil[1]{School of Mathematical Sciences, The University of Adelaide}
\affil[2]{DTU Compute, Technical University of Denmark}
\date{\empty}
\maketitle

\begin{abstract} 
We construct a stochastic fluid process with an underlying piecewise deterministic Markov process (PDMP) akin to the one used in the construction of the rational arrival process (RAP) in \cite{asm:99}, which we call the RAP-modulated fluid process. As opposed to the classic stochastic fluid process driven by a Markov jump process, the underlying PDMP of a RAP-modulated fluid process has a continuous state space and is driven by matrix parameters which may not be related to an intensity matrix. Through novel techniques we show how well-known formulae associated to the classic stochastic fluid process, such as first passage probabilities and the stationary distribution of its queue, translate to its RAP-modulated counterpart.

\end{abstract} 

\section{Introduction}

A \emph{stochastic fluid model}  $(\mathcal{X}, \varphi) = \{(X_t, \varphi_t)\}_{t \geq 0}$ is a Markov additive process in which the background component $\varphi$ is a Markov jump process with finite state space and the additive component $\mathcal{X}$ is piecewise linear, with a rate that depends on the background state:
	\begin{align*} 
		X_t = \int_0^t r_{\varphi_s} \dd s, \quad t\ge 0. 
	\end{align*} 
The study of steady-state aspects of stochastic fluid models goes back to \cite{Rogers:1994uoa, Asmussen:1995jm, Karandikar:1995vo}, and its literature has been prolific from both applied and theoretical perspectives. Latouche and Nguyen~\cite{latouche2018analysis} provide a comprehensive survey on the topic; much of the existing analysis exploits the probabilistic interpretation of the background process $\varphi$ with finite state space.


A related concept is a \emph{Markovian arrival process} $(\mathcal{N}, \mathcal{J}) = \{(N_t, J_t)\}_{t \geq 0}$, where $\mathcal{N}$ is a counting process and $\mathcal{J}$ is the underlying Markov jump process with initial distribution $\bm{\alpha}$, hidden jumps according to intensity matrix $C$, and observable jumps according to according to a nonnegative intensity matrix~$D$. 
%
%
%
Jump times of the latter kind correspond to the arrival epochs of $\mathcal{N}$. 
%
%
Asmussen and Bladt \cite{asm:99} introduce \emph{rational arrival processes} (RAP), a generalisation of Markovian arrival processes for which 
 $\bm{\alpha}$, 
 $C$ and $D$ are not necessarily related to the parameters of a Markov jump process. They show that the arrivals associated to this algebraic generalisation are determined by an underlying continuous-state-space piecewise-deterministic Markov process which here we refer to as \emph{an orbit process}. This orbit process evolves  deterministically between its jump times, which occur according to a given space-dependent intensity function. The arrival epochs of the RAP can then be regarded as the jump times of the orbit process, allowing for a probabilistic analysis of the former using the physical interpretation of the latter.
 %
%
%
%

In the present paper, we consider an extension of stochastic fluid models, which we call \emph{RAP-modulated fluid processes}. 
A RAP-modulated fluid process $(\mathcal{R}, \orbit) = \{(R_t, \boldsymbol{A}_t)\}_{t \geq 0}$ is a Markov additive process in which the additive component $\mathcal{R}$ is still piecewise linear but the background component $\orbit$ is now an orbit process. A similar generalisation was introduced in \cite{Bean:2010jc}, from  \emph{Quasi-Birth-Death processes}, where the additive component lives on the integers and is modulated by a Markovian jump process, to \emph{QBD-RAP}, where the additive component still lives on the integers but is now modulated by a RAP. The class of RAP-modulated fluid processes we introduce here goes beyond stochastic fluid processes; for instance, one may use a Markov renewal process with matrix-exponential times in order to model the orbit process of a RAP-modulated fluid process. 

Additional to defining the class of RAP-modulated fluid processes, the contributions of this paper include providing first passage probabilities, an expression for the stationary distribution of its queue, and an algorithm for computing the expected value of the orbit process $\boldsymbol{\mathcal{A}}$ at the first downcrossing times of the level process to level $0$.  

The structure of the paper is as follows. We define in Section~\ref{sec:proto-FRAP} the simplest RAP-modulated fluid process, which we call a \emph{simple RAP-modulated fluid process}. Although simplistic in its nature, this process helps us introduce the physical analysis and discuss necessary conditions of a RAP-modulated fluid process. In Section~\ref{sec:generalmodel}, we give a precise definition of a RAP-modulated fluid process and present first passage results which exemplify the framework and techniques used throughout the manuscript. Then, we focus on the case in which the level process $\mathcal{R}$ of the RAP-modulated fluid process is nowhere-constant in Section~\ref{sec:nozerorates}, and allow $\mathcal{R}$ to be piecewise-constant in Section~\ref{sec:withzerorates}. 
%
%

Our work demonstrates that several results of stochastic fluid models
translate into the framework of RAP-modulated fluid processes, although significantly more complicated techniques are required to analyze the latter.  

\section{Simple RAP-modulated fluid process}
	\label{sec:proto-FRAP}
In the following we introduce the simplest non-trivial example of a RAP-modulated fluid process, $(\mathcal{R}, \orbit)=\{(\mathcal{R}_t, \bm{A}_t )\}_{t \geq 0}$, which we refer to as \emph{a simple RAP-modulated fluid process}. To that end, first we define $\orbit=\{\bm{A}_t\}_{t\ge 0}$, the \emph{(simple) orbit process}. 

The process $\orbit$ is a c\`adl\`ag piecewise-deterministic Markov process (PDMP) with state space $\cup_{k\in\{+,-\}}\mathfrak{Z}^k$, where for each $k\in\{+,-\}$, $\mathfrak{Z}^k$ is a subset of the affine hyperplane $\{\boldsymbol{x}\in\mathds{R}^{m^k}:\boldsymbol{x}\bm{1}=1\}$ for some fixed $m^k\ge 1$. Here, the elements of $\mathfrak{Z}^k$ are regarded as row vectors and $\bm{1}$ is a column vector of ones of appropriate dimension. Thus, $\orbit$ is a row{-}vector process of varying dimension, either $m^+$ or $m^-$ depending whether it is in $\mathfrak{Z}^+$ or $\mathfrak{Z}^-$ at the given instant. In general $m^+\neq m^-$, but even in the case $m^+=m^-$ the subsets $\mathfrak{Z}^+$ and $\mathfrak{Z}^-$ will be considered to belong to different spaces, and thus they will always be disjoint. We let the initial point $\bm{A}_0$ be arbitrary but fixed.

Each PDMP is characterized by its motion between jumps, jump intensity and transition mechanism at its jump epochs \cite{Davis:1984vi}, properties which we define for $\orbit$ next. During an interval without jumps, say $[t,t+h)$ with $h> 0$, the orbit process $\orbit$ evolves according to the system of ordinary differential equations (ODE) given by 
\begin{align}
	\label{eq:ODE1}
	\frac{\dd \bm{A}_s}{\dd s}= \left\{\begin{array}{cl}
	\bm{A}_sC^+ - (\bm{A}_sC^+\bm{1})\bm{A}_s&\mbox{for } \bm{A}_s\in\mathfrak{Z}^+,\\
	\vspace*{-0.3cm} \\
	\bm{A}_sC^- - (\bm{A}_sC^-\bm{1})\bm{A}_s&\mbox{for }  \bm{A}_s\in\mathfrak{Z}^-, \quad s\in (t,t+h),
\end{array}\right. 
\end{align}
for some $C^+\in \mathds{R}^{m^+\times m^+}$ and  $C^- \in \mathds{R}^{m^-\times m^-}$. The solution to the ODE (\ref{eq:ODE1}) is
\begin{align}
	\label{eq:solODE1}
		\bm{A}_{t+r} = \left\{\begin{array}{cl} 
  \displaystyle\frac{\bm{A}_te^{C^+r}}{\bm{A}_te^{C^+r}\bm{1}}\in\mathfrak{Z}^+&\mbox{if } \bm{A}_t\in\mathfrak{Z}^+, \\
  \vspace*{-0.2cm} \\
  \displaystyle\frac{\bm{A}_te^{C^-r}}{\bm{A}_te^{C^-r}\bm{1}} \in\mathfrak{Z}^-&\mbox{if } \bm{A}_t\in\mathfrak{Z}^-,\quad r\in [0,h).
  \end{array}\right.
  	\end{align}
  Note that in (\ref{eq:solODE1}) we implicitly assume that, for each initial point in $\mathfrak{Z}^+\cup\mathfrak{Z}^-$, the system of ODEs (\ref{eq:ODE1}) evolves entirely within $\mathfrak{Z}^+\cup\mathfrak{Z}^-$. 

As $\orbit$ evolves within $\mathfrak{Z}^+\cup\mathfrak{Z}^-$, jump epochs occur according to a location-dependent intensity function $\lambda:\mathfrak{Z}^+\cup\mathfrak{Z}^-\mapsto\mathds{R}_+$ given by
\begin{align}
	\label{eq:jumpint1}
	\lambda(\bm{A}_t)  = \left\{\begin{array}{cc}
		\bm{A}_tD^{+-}\bm{1}&\mbox{if } \bm{A}_t \in \mathfrak{Z}^+,\\ 
		\vspace*{-0.3cm} \\
		\bm{A}_t D^{-+}\bm{1}&\mbox{if } \bm{A}_t \in \mathfrak{Z}^-,
	\end{array}\right.
\end{align}
for some $D^{+-}\in \mathds{R}^{m^+\times m^-}$ and $D^{-+}\in \mathds{R}^{m^-\times m^+}$. Since $\lambda(\cdot)$ is assumed to be a nonnegative function, it indeed corresponds to a valid intensity function.             

\begin{Condition} 
	$C^+\bm{1} + D^{+-}\bm{1} = \bm{0}$ and $C^-\bm{1} + D^{-+}\bm{1}=\bm{0}$.
\end{Condition} 

This condition implies that $\lambda$ can alternatively be written as
\begin{align}
  \label{eq:jumpint5}
  \lambda(\bm{A}_t)  = \left\{\begin{array}{cc}
    -\bm{A}_tC^{+}\bm{1}&\mbox{if } \bm{A}_t \in \mathfrak{Z}^+,\\ 
    \vspace*{-0.3cm} \\
    -\bm{A}_t C^{-}\bm{1}&\mbox{if } \bm{A}_t \in \mathfrak{Z}^-.
  \end{array}\right.
\end{align}
Using (\ref{eq:solODE1}) and (\ref{eq:jumpint5}) it can be readily verified that
\begin{align}
	\label{eq:nojumps6}
	\mathds{P}\left(\left.\mbox{$\orbit$ has no jumps in $[t, t+h]$}\;\right|\bm{A}_t\right)= \left\{\begin{array}{cc}
  \bm{A}_te^{C^+h}\bm{1}&\mbox{if } \bm{A}_t\in\mathfrak{Z}^+,\\
  \vspace*{-0.2cm} \\
\bm{A}_te^{C^-h}\bm{1}&\mbox{if } \bm{A}_t\in\mathfrak{Z}^-.\end{array}\right.
\end{align}
Indeed, {by} 
\cite{Davis:1984vi}, the function 
%
	$F(h):=\mathds{P}\left(\left.\mbox{$\orbit$ has no jumps in $[t, t+h]$}\;\right|\bm{A}_t, \bm{A}_t\in\mathfrak{Z}^k\right)$
%
corresponds to the unique differentiable solution of
\begin{align}
	\label{eq:probnojumpsaux3}
	\log F(h)=-\int_0^h \lambda\left(\frac{\bm{A}_te^{C^kr}}{\bm{A}_te^{C^kr}\bm{1}}\right)\dd r = \int_0^h \left[\frac{\bm{A}_te^{C^kr}}{\bm{A}_te^{C^kr}\bm{1}}\right]C^k\bm{1}\dd r.
	\end{align}
That $F(h) = \bm{A}_te^{C^kh}$ is a solution follows by differentiating both sides of (\ref{eq:probnojumpsaux3}).

Finally, 
given that a jump occurs at time $t$, the orbit process $\orbit$ will directly jump to
\begin{align}
	\label{eq:detjumps3}
	\frac{\bm{A}_{t^-} D^{+-}}{\bm{A}_{t^-}D^{+-}\bm{1}}\in\mathfrak{Z}^- \mbox{ if } \bm{A}_{t^-}\in\mathfrak{Z}^+,\quad \mbox{and} \quad 
\frac{\bm{A}_{t^-}D^{-+}}{\bm{A}_{t^-}D^{-+}\bm{1}}\in\mathfrak{Z}^+ \mbox{ if }  \bm{A}_{t^-}\in\mathfrak{Z}^-,
\end{align}
so that a jump originating from $\bm{x} \in\mathfrak{Z}^+$ will land at a deterministic point in $\mathfrak{Z}^-$, and vice versa.
%

%
%
%
Note that the motion between jumps and jump behaviour of the PDMP $\orbit$ depend only on the matrices $C^+, C^-, D^{+-}$ and $D^{-+}$. Moreover, these matrices implicitly dictate some requirements on the state space $\mathfrak{Z}^+\cup\mathfrak{Z}^-$ through Equations (\ref{eq:solODE1}), (\ref{eq:jumpint1}) and (\ref{eq:detjumps3}). Verifying if a set of matrices is \emph{compatible} with a given state space is by no means trivial, which is a known issue in the context of matrix--exponential distributions and RAPs \cite{bladt2017matrix}. Below we present two examples of \emph{valid} orbit processes.

\begin{Example}[Markov jump process]\label{ex:MJP1}
Let $C^+$ and $C^-$ be subintensity matrices, and let $D^{+-}$ and $D^{-+}$ be nonnegative matrices, in such a way that
\[\begin{pmatrix}
C^+&D^{+-}\\ 
D^{-+}& C^-
\end{pmatrix}\]
corresponds to the intensity matrix of a Markov jump process. For $k\in\{+,-\}$ choose
\[\mathfrak{Z}^k=\{\bm{x}\in\mathds{R}^{m^k}: \bm{x}\ge \bm{0}, \bm{x}\bm{1}=1\}.\]
Note that for all $\bm{b}\in\mathfrak{Z}^k$ and $h\ge 0$, ${\bm{b}e^{C^kh}}/{\bm{b}e^{C^kh}\bm{1}}$ corresponds to a normalized probability vector, and thus belongs to $\mathfrak{Z}^k$. The nonnegativity of $D^{+-}$ and $D^{-+}$ implies that the intensity function $\lambda$ is always nonnegative. Finally, if {$k\ne\ell\in\{+,-\}$} and $\bm{b}\in\mathfrak{Z}^k$, then $\bm{b}D^{k\ell}/{\bm{b}D^{k\ell}\bm{1}}$ corresponds to an $m^\ell$--dimensional normalized probability vector, and thus belongs to $\mathfrak{Z}^\ell$. Thus, an orbit process defined by these parameters must be valid.
\end{Example}
\begin{Example}[Matrix-exponential renewal process] \label{ex:ME1}For $k\in\{+,-\}$, let $(\bm{\alpha}^k,S^k)$ be such that $\bm{\alpha}^k\bm{1} = 1 $ and $F^k(x):=1-\bm{\alpha^k}e^{S^kx}\bm{1}$ is a distribution function. Such a class of distributions is commonly known as matrix-exponential, a generalization of phase--type distributions. For $\ell\neq k$, let 
\[C^k=S^k\quad \mbox{and}\quad D^{k\ell} = (-\bm{S}^k\bm{1})\bm{\alpha}^\ell,\]
and define
\[\mathfrak{Z}^k=\left\{\frac{\bm{\alpha}^ke^{C^kh}}{\bm{\alpha^k}e^{C^kh}\bm{1}}: h\ge 0\right\}.\]
If $\bm{b}\in \mathfrak{Z}^k$, then $\bm{b}={\bm{\alpha}^ke^{C^kh}}/{\bm{\alpha^k}e^{C^kh}\bm{1}}$ for some $h\ge 0$ and thus
\begin{align}
	\label{eq:lambdaaux4}
	\lambda(\bm{b}) = \bm{b}D^{k\ell}\bm{1} = \frac{\bm{\alpha}^ke^{S^kh}}{\bm{\alpha^k}e^{S^kh}\bm{1}}\left((-\bm{S}^k\bm{1})\bm{\alpha}^\ell\right)\bm{1}=\frac{\bm{\alpha}^ke^{S^kh}(-\bm{S}^k\bm{1})}{\bm{\alpha^k}e^{S^kh}\bm{1}}.
	\end{align}
Since the numerator in the last expression of (\ref{eq:lambdaaux4}) corresponds to a probability density function and its denominator to a tail distribution, $\lambda$ is indeed a nonnegative intensity function. Moreover, if a jump happens from some $\bm{b}\in\mathfrak{Z}^k$, it will land in 
\[\frac{\bm{b}D^{k\ell}}{\bm{b}D^{k\ell}\bm{1}}=\frac{\bm{b}(-\bm{S}^k\bm{1})\bm{\alpha}^\ell}{\bm{b}(-\bm{S}^k\bm{1})\bm{\alpha}^\ell\bm{1}}=\frac{\bm{\alpha}^\ell}{\bm{\alpha}^\ell\bm{1}}=\bm{\alpha}^\ell\in\mathfrak{Z}^\ell,\quad \ell\neq k.\]
Thus, this set of parameters corresponds to a valid orbit process whose {inter}-jump times are matrix-exponential.
\end{Example}

Now that the distributional characteristics of the orbit process have been completely described, we state two further conditions on $\orbit$ and its state space. 
\begin{Condition}
	\label{cond:boundedness} 
	The sets $\mathfrak{Z}^+$ and $\mathfrak{Z}^-$ are bounded. 
\end{Condition} 

In the RAP setting of \cite{asm:99}, Condition~\ref{cond:boundedness} is needed for $\boldsymbol{\mathcal{A}}$ to correspond to the coefficients of a linear combination of probability measures that is itself a probability measure. In our setting, it will enable us to use the Bounded Convergence Theorem in specific places and to guarantee that, almost surely, there are only finitely many jumps on compact time intervals.

\begin{Condition}
	\label{cond:minimal} 
	For $k\in\{+,-\}$, the set $\mathfrak{Z}^k$ is contained in a \emph{minimal} $(m^k-1)$-dimensional affine hyperplane, meaning that $\mathfrak{Z}^k$ cannot be contained in any $(m^k -2)$--dimensional affine hyperplane. %
\end{Condition} 

In the context of Example \ref{ex:MJP1}, Condition~\ref{cond:minimal} is equivalent to an irreducibility assumption of the Markov jump process, while in the context of Example \ref{ex:ME1} it is linked to the minimal dimension property of matrix-exponential distributions \cite{bladt2017matrix}. In our context, Condition~\ref{cond:minimal} tells us that $\mathfrak{Z}^k$ is \emph{rich enough} to guarantee a one-to-one correspondence between a matrix $G$ and the collection $\{\bm{b}G: \bm{b}\in \mathfrak{Z}^+\}$; this property follows from Lemma \ref{lem:minimality} below.

\begin{Lemma}\label{lem:minimality}
For $k\in\{+,-\}$, there exist $m^k$ linearly independent vectors contained in $\mathfrak{Z}^k$.
\end{Lemma}
\begin{proof}
Since $\mathfrak{Z}^k$ is contained in a minimal $(m^k -1)$-dimensional affine hyperplane, there exist linearly independent vectors $\bm{b}_1, \dots, \bm{b}_{m^k-1}$ and $\bm{g}\in \mathfrak{Z}^k$ such that $\{\bm{b}_i\}$ are linearly independent from $\bm{g}$ with
\[\{\bm{b}_1, \dots, \bm{b}_{m^k-1}\} + \bm{g} \subset\mathfrak{Z}^+\subset\mbox{span}\{\bm{b}_1, \dots, \bm{b}_{m^k-1}\} + \bm{g}.\]
Thus, let $\bm{h}_{m^k}= \bm{g}$ and $\bm{h}_i= \bm{b}_i + \bm{g}$ for $ i\in\{1,\dots,m^k-1\}$. Then $\{\bm{h}_i\}\subset\mathfrak{Z}^k$ is a collection of $m^k$ linearly independent vectors.
\end{proof}
\begin{definition} 
	Let $\nu^+ \in (0, \infty)$, $\nu^- \in (-\infty, 0)$. A \emph{simple RAP-modulated fluid process} is a Markov additive process $\{(R_t, \bm{A}_t)\}_{t\ge 0}$ with an orbit process $\orbit=\{\bm{A}_t\}_{t\ge 0}$ and additive component $\mathcal{R}=\{R_t\}_{t\ge 0}$ of the form
	\begin{align} 
	R_t=\int_0^t {\nu}^+\Indic{\bm{A}_s\in\mathfrak{Z}^+} + {\nu}^-\Indic{\bm{A}_s\in\mathfrak{Z}^-} \dd s.
	\end{align} 
  We refer to $\mathcal{R}$ as the \emph{level process}.
\end{definition} 

In other words, at time $t$ the process $\mathcal{R}$ is either increasing at rate $\nu^+$ or decreasing at rate $\nu^-$, depending on the location of $\bm{A}_t$. The term \emph{simple} stems from the facts that $\mathcal{R}$ is nowhere piecewise constant, and that $\orbit$ is allowed to perform jumps from $\mathfrak{Z}^+$ to $\mathfrak{Z}^-$ or vice versa only. We relax these assumptions in Section \ref{sec:generalmodel}. See Figure \ref{fig:simpleFRAP} for a visual description of the simple RAP-modulated fluid process.
 
\begin{figure}[h]
\centering
\begin{tikzpicture}[scale=0.63]
\node  (0) at (-3.25, 1) {};
    \node[circle, fill=green, very thick, minimum size=7mm]  (1) at (-1.75, -1.75) {};
    \node  (2) at (-2, 0.25) {};
    \node  (3) at (-2.75, 0.5) {};
    \node  (5) at (-3, 2.75) {};
    \node  (6) at (-1.75, -3) {};

        \filldraw [color=gray!10, in=-165, out=-165, looseness=1.25] (5.center) to (6.center);
    \filldraw [color=gray!10, bend left=90, looseness=1.25] (5.center) to (6.center);
        \draw [->, very thick, color=blue, in=75, out=105, looseness=2.00] (2.center) to (3.center);
    \draw [->, very thick, color=blue, in=-135, out=-120, looseness=1.00] (0.center) to (1.center);
        \node  (4) at (-2.5, 3.5) {$\mathfrak{Z}^+$};
        \node  (A0) at (-3.25, 1.35) {\footnotesize$\bm{A}_{0}$};
        \node  (A2) at (-1.85, -0.07) {\footnotesize$\bm{A}_{t_2}$};
        \node  (0d) at (5, 2.75) {};
    \node (1d) at (2, -3) {};
    \node  (2d) at (2, -0.5) {};
    \node (3d) at (4.75, -0.25) {};
    \node  (4d) at (4, -0.5) {};
    \node  (5d) at (3.5, 0.75) {};
    \filldraw [color=gray!10, bend left=90, looseness=1.00] (0d.center) to (1d.center);
    \filldraw [color=gray!10, bend right=90, looseness=1.00] (0d.center) to (1d.center);
    \draw [<-, very thick, color=red, in=-105, out=-105, looseness=2.50] (2d.center) to (3d.center);
    \draw [->, very thick, color=red, in=60, out=45, looseness=2.75] (5d.center) to (4d.center);
    \node  (6d) at (4.2, 3.5) {$\mathfrak{Z}^-$};
        \node (A3) at (5, 0.05) {\footnotesize$\bm{A}_{t_3}$};
    \node  (A1) at (3.3, 0.45) {\footnotesize$\bm{A}_{t_1}$};
   
\end{tikzpicture}
\qquad
\begin{tikzpicture}[scale=0.8]
    \node  (0) at (0, 5) {};
    \node  (1) at (0, 0) {};
    \node (2) at (4.75, 0) {};
    \node (3) at (2, 4) {};
    \node (4) at (2.75, 2.5) {};
    \node (5) at (3.25, 3.5) {};
    \node (6) at (4.75, 0.5) {};
    \node (7) at (4.95, 0) {$t$};
    \node (8) at (0, 5.4) {$R_t$};
    \node (t1p) at (2,0) {};
    \node (t2p) at (2.75,0) {};
    \node (t3p) at (3.25,0) {};
    \node (t0) at (0,-0.35) {$0$};
    \node (t1) at (2,-0.35) {$t_1$};
    \node (t2) at (2.75,-0.35) {$t_2$};
    \node (t3) at (3.25,-0.35) {$t_3$};
        \draw [<-](0.center) to (1.center);
    \draw [->](1.center) to (2.center);
    \draw [thick, color=blue](1.center) to (3.center);
    \draw [thick, color=red] (3.center) to (4.center);
    \draw [thick, color=blue] (4.center) to (5.center);
    \draw [thick, color=red](5.center) to (6.center);
    \draw [thin, dashed](3.center) to (t1p.center);
    \draw [thin, dashed](4.center) to (t2p.center);
    \draw [thin, dashed](5.center) to (t3p.center);

\end{tikzpicture}
\caption{A sample path of the elements of the simple RAP-modulated fluid process $\{(R_t, \bm{A}_t)\}_{t\ge 0}$. The times $t_1, t_2, t_3$ correspond to jump epochs of the orbit process $\bm{\mathcal{A}}$. In between jumps, $\{\bm{A}_t\}_{t\ge 0}$ evolves deterministically, switching between states in $\mathfrak{Z}^+$ and $\mathfrak{Z}^-$ at $t_1, t_2$ and $t_3$.}
\label{fig:simpleFRAP}
\end{figure}
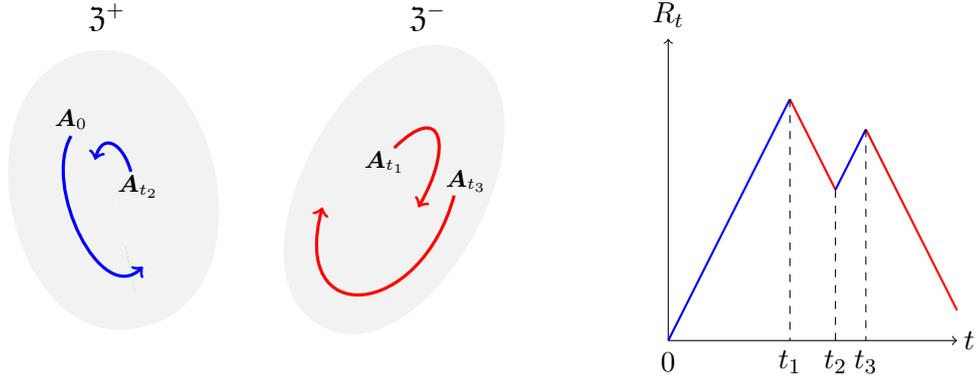
%
%

\begin{Condition} 
For $k\in \{+, -\}$, 
$ 
\lim_{t\rightarrow\infty}\mathds{P}\left(\left.\bm{A}_s\in \mathfrak{Z}^{k} \;\forall s\in [0,t]\;\right| \bm{A}_0 =\bm{\alpha}\right) = 0$  for all $\bm{\alpha}\in\mathfrak{Z}^{k}.
$
\end{Condition} 
This implies that neither the orbit $\orbit$ nor the level $\mathcal{R}$ are deterministic and thus trivial.
\section{RAP-modulated fluid process}
	\label{sec:generalmodel}
	\subsection{Definition} 
%
In the following we introduce the RAP-modulated fluid process, which allows for greater generality than the simple RAP-modulated fluid process: in addition to allowing $\{R_t\}_{t\ge 0}$ to be piecewise constant, the orbit $\orbit$ is also allowed to perform jumps within each set $\mathfrak{Z}^k$, $k\in\{+,-,0\}$, via a specific set-partitioning explained next.

Denote $\mathcal{S}=\{+,-,0\}$ and fix $k\in\mathcal{S}$. We suppose that the set $\mathfrak{Z}^k$ is contained in a collection of $n^k\ge 0$ orthogonal affine hyperplanes. Specifically, we assume that the set $\mathfrak{Z}^k$ can be partitioned in sets, say $\{\mathfrak{Z}^k_i\}_{i=1}^{n^k}$, with the following property: \magenta{T}here exists a collection $\{m^k_i\}_{i=1}^{n^k}\subset\{1,2,\dots\}$ such that each set $\mathfrak{Z}^k_i$ is contained in the affine hyperplane
\begin{align}
	\label{eq:affhyp1}
	\left\{\boldsymbol{x}\in\mathds{R}^{\eta^k}: \boldsymbol{x}=(\boldsymbol{0}_{1}^k,\dots, \boldsymbol{0}_{i-1}^k, \boldsymbol{y},\boldsymbol{0}_{i+1}^k,\dots,\boldsymbol{0}_{{n^k}}^k) \mbox{ for some }\boldsymbol{y}\in\mathds{R}^{m^k_i}\mbox{ with }\boldsymbol{y}\bm{1}=1\right\},\end{align}
where $\bm{0}_i^k$ denotes the row-vector of zeros with $m^k_i$ elements and $\eta^k:=\sum_{i=1}^{n^k}m^k_i$ corresponds to the dimension of the space in which $\mathfrak{Z}^k$ lives.

The \emph{(general) orbit process} $\orbit$ is a c\`adl\`ag PDMP with state space $\mathfrak{Z} = \cup_{k\in\mathcal{S}}\mathfrak{Z}^k$ and some arbitrary but fixed initial state $\bm{A}_0$. We describe its PDMP characteristics next. If $\orbit$ has no jumps in $[t,t+h)$ for some $h>0$, then $\orbit$ follows the ODE
\begin{align}\label{eq:ODEFRAP1}
\frac{\dd \bm{A}_s}{\dd s}= \bm{A}_s\Gamma^k - (\bm{A}_s\Gamma^k\bm{1})\bm{A}_s\quad\mbox{for } \bm{A}_s\in\mathfrak{Z}^k, s\in[t,t+h],
\end{align}
where $\Gamma^k \in\mathds{R}^{\eta^k\times{\eta^k}}$ is of the form
\[\Gamma^k
	=\begin{pmatrix}
	C^k_{11}&&&\\
	&C^k_{22}&&\\
	&&\ddots&\\
	&&&&C^k_{n^kn^k}
\end{pmatrix}
\]
%
for some $C^k_{ii}\in\mathds{R}^{m^k_i\times m^k_i}$, $1\le i\le n^k$. The block-diagonal structure of $\Gamma^k$ guarantees that, if $\bm{A}_t\in\mathfrak{Z}^k_{i_0}$ for some $1\le i_0\le n^k$ and $\orbit$ has no jumps in $[t,t+h]$, then $\{\bm{A}_s\}_{s=t}^{t+h}$ is contained in the affine hyperplane (\ref{eq:affhyp1}) with $i=i_0$. 
One can verify that the solution of (\ref{eq:ODEFRAP1}) is given by
\begin{align}
  \label{eq:Atnojumpsh2}
  \bm{A}_{t+r}=\frac{\bm{A}_t e^{\Gamma^k r}}{\bm{A}_t e^{\Gamma^k r}\bm{1}}\in\mathfrak{Z}^k\quad\mbox{if } \bm{A}_{t}\in\mathfrak{Z}^k, r\in[0,h].
\end{align}

If $\boldsymbol{A}_t \in \mathfrak{Z}^k_i$, then a jump to $\mathfrak{Z}^k_j$ with $j\neq i$ occurs with intensity $\boldsymbol{A}_t \widehat{C}^k_{ij}\bm{1}\ge 0$ with
%
%
\[\widehat{C}^k_{ij}:=\begin{pmatrix}0^k_{1,1}&\cdots&0^k_{1,j-1}& &  & & \\
\vdots&&\vdots& && & \\
0^k_{i-1,1}&\cdots&0^k_{i-1,j-1}& && & \\
 & & &C^k_{ij}& &  & \\
&&& &0^k_{i+1,j+1}&\cdots & 0^k_{i+1,n^k}\\
&&&& \vdots & & \vdots\\
&&&&0^k_{n^k,j+1}&\cdots & 0^k_{n^k,n^k}
\end{pmatrix}\]
for some $C^k_{ij}\in\mathds{R}^{m^k_i\times m^k_j}$, where $0^k_{a,b}$ denotes the zero matrix in $\mathds{R}^{m^k_a\times m^k_b}$. {If such a jump to $\mathfrak{Z}^k_j$ occurs at some time $s>t$, it will land at $\bm{A}_s=(\bm{A}_{s^-}\widehat{C}_{ij}^k)/(\bm{A}_{s^-}\widehat{C}_{ij}^k\bm{1})\in\mathfrak{Z}^k_j$.} 

Similarly, if $\bm{A}_t \in \mathfrak{Z}^k_i$, then a jump to $\mathfrak{Z}^\ell_j$ with $\ell\neq k, j\le n^\ell$, occurs with intensity $\boldsymbol{A}_t \widehat{D}^{k\ell}_{ij}\bm{1}\ge 0$ where
%
%
\[\widehat{D}^{k\ell}_{ij}=\begin{pmatrix}0^{k\ell}_{1,1}&\cdots&0^{k\ell}_{1,j-1}& & & & \\
\vdots&&\vdots&&& & \\
0^{k\ell}_{i-1,1}&\cdots&0^{k\ell}_{i-1,j-1}& & &  & \\
& & &D^{k\ell}_{ij}& &  & \\
& &&&0^{k\ell}_{i+1,j+1}&\cdots & 0^{k\ell}_{i+1,n^k}\\
 & & & &\vdots&  & \vdots\\
 & & &&0^{k\ell}_{n^k,j+1}&\cdots & 0^{k\ell}_{n^k,n^k}
\end{pmatrix}\]
for some $D^{k\ell}_{ij}\in\mathds{R}^{m^k_i\times m^\ell_j}$. 
{If such a jump 
occurs at some time $s>t$, it will land at $\bm{A}_s=(\bm{A}_{s^-}\widehat{D}_{ij}^{k\ell})/(\bm{A}_{s^-}\widehat{D}_{ij}^{k\ell}\bm{1}) \in\mathfrak{Z}^\ell_j$.}
%

\begin{Condition} 
\begin{align}
	\label{eq:zerosum1}\Gamma^k\bm{1} +\sum_{i=1}^{n^k} \sum_{j\neq i}\widehat{C}^k_{ij}\bm{1} + \sum_{\ell\neq k}\sum_{i=1}^{n^\ell} \sum_{j\neq i}\widehat{D}^{k\ell}_{ij}\bm{1}=\bm{0}.\end{align}
\end{Condition} 
Since the jump intensities of a PDMP are additive, (\ref{eq:zerosum1}) implies that the jump intensity function of $\orbit$ is given by $\lambda:\cup_{k\in\mathcal{S}}\mathfrak{Z}^k \mapsto \mathds{R}_+$ of the form
\begin{align}\label{eq:lambda2}\lambda(\bm{A}_t)=\bm{A}_t\left(\sum_{i=1}^{n^k} \sum_{j\neq i}\widehat{C}^k_{ij}\bm{1} + \sum_{\ell\neq k}\sum_{i=1}^{n^\ell} \sum_{j\neq i}\widehat{D}^{k\ell}_{ij}\bm{1}\right)=-\bm{A}_t\Gamma^k\bm{1}\quad\mbox{for}\quad\bm{A}_t\in\mathfrak{Z}^k.\end{align}
Analogous to (\ref{eq:nojumps6}), Equations (\ref{eq:Atnojumpsh2}) and (\ref{eq:lambda2}) imply that
\begin{align}
	\label{eq:genorbitnojumps1}
	\mathds{P}\left(\left.\orbit\mbox{ has no jumps in }[t,t+h]\;\right| \bm{A}_t,\bm{A}_t\in\mathfrak{Z}^k_i\right) = \bm{A}_t e^{\Gamma^k h}\bm{1}.
\end{align}
As in Section \ref{sec:proto-FRAP}, note that the distributional properties of $\orbit$ are completely determined by 
$\{C^k_{ij}: k\in\mathcal{S}, i\le n^k, j\le n^k\}$ and $\{D^{k\ell}_{ih}: k\in\mathcal{S}, \ell\neq k, i\le n^k, h\le n^\ell\}.$ 
Similarly, here it is also difficult to assess if a collection of matrices is compatible with a given state space $\cup_{k\in\mathcal{S}}\cup_{i\le n^k}\mathfrak{Z}^k_i$, nonetheless, the following is an example of a valid orbit process.
\begin{Example}[Markov renewal matrix-exponential process] For $k,\ell\in\mathcal{S}$, let $P^{k\ell}=\{p^{k\ell}_{ij}\}_{ij}\in\mathds{R}^{n^k\times n^\ell}$ be such that
\[
	\begin{pmatrix}P^{++}&P^{+-}&P^{+0}\\P^{-+}&P^{--}&P^{-0}\\P^{0+}&P^{0-}&P^{00}
\end{pmatrix}
\]
is a transition probability matrix. For $k\in\mathcal{S}$ let $\{(\bm{\alpha}^k_i, S^k_i)\}_{i=1}^{n^k}$ be a collection of matrix-exponential parameters. In a similar fashion to Example \ref{ex:ME1}, one can verify that taking
\[
	C^k_{ii} := S^k_i,\quad C^k_{ij}:=p^{kk}_{ij}\left(-S^k_i\bm{1}\right)\bm{\alpha}^k_j,\quad D^{k\ell}_{ih}:=p^{k\ell}_{ih}\left(-S^k_i\bm{1}\right)\bm{\alpha}^\ell_h\quad\mbox{for}\quad\ell\neq k, {i\neq j\le n^k}, h\le n^\ell,
	\]
and defining
\[
	\mathfrak{Z}^k_i := \left\{\boldsymbol{x}\in\mathds{R}^{\eta^k}: \boldsymbol{x}=\left(\boldsymbol{0}_{1}^k,\dots, \boldsymbol{0}_{i-1}^k, \frac{\bm{\alpha}^ke^{C^k_{ii}h}}{\bm{\alpha^k}e^{C^k_{ii}h}\bm{1}},\boldsymbol{0}_{i+1}^k,\dots,\boldsymbol{0}_{{n^k}}^k\right) \mbox{ for some }h\ge 0 \right\}
	\]
yields a valid orbit $\orbit$ driven by a Markov renewal process with matrix-exponential jump times.
\end{Example}

\begin{Condition} 
For each $k\in\{+,-,0\}$ and $i\le n^k$,  the set $\mathfrak{Z}^k_i$ is bounded and is contained in a minimal $(m^k_i-1)$-dimensional affine hyperplane.
\end{Condition} 
The above condition 
implies the following.
 
\begin{Lemma}
	\label{lem:bminimality} 
	For $k\in\{+,-,0\}$, the set $\mathfrak{Z}^k$ contains $\eta^k$ linearly independent vectors.
\end{Lemma}
\begin{definition}
Let $\{\nu^+_j\}_{j=1}^{n^+}\subset(0,\infty)$ and $\{\nu^-_j\}_{j=1}^{n^-}\subset(-\infty,0)$. We define a \emph{RAP-modulated fluid process} to be a Markov additive process  $(\mathcal{R},\orbit)=\{(R_t, \bm{A}_t)\}_{t\ge 0}$ with $\orbit$ a \emph{(general) orbit process} and $\mathcal{R}$ a \emph{(general) level process} of the form
\begin{align}
	R_t := \int_0^t \left(\sum_{j=1}^{n^+} \nu^+_j \Indic{\bm{A}_s\in\mathfrak{Z}^+_j} + \sum_{j=1}^{n^-} \nu^-_j \Indic{\bm{A}_s\in\mathfrak{Z}^-_j}\right) \dd s.
\end{align} 
\end{definition}
See Figure \ref{fig:multijumpFRAP} for a visual description. 

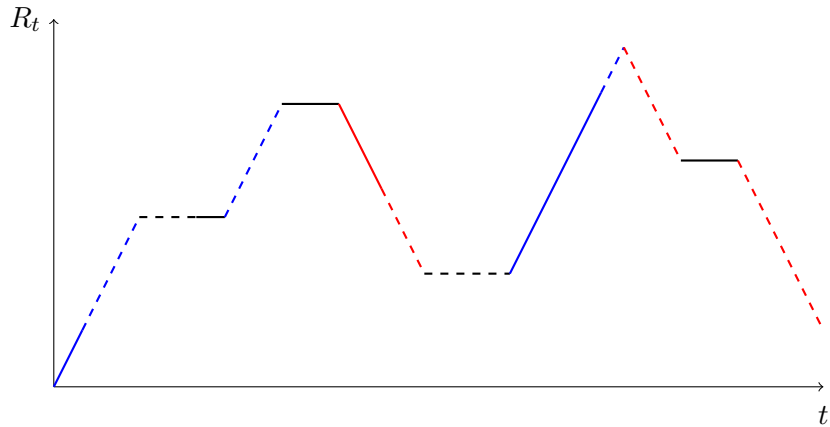
\begin{figure}[h]
\centering
\begin{tikzpicture}[scale=1.5]
\node  (0) at (0, 0) {};
    \node  (1) at (6.75, 0) {};
    \node  (2) at (0, 3.25) {};
    \node  (3) at (0.25, 0.5) {};
    \node  (4) at (0.75, 1.5) {};
    \node  (5) at (1.25, 1.5) {};
    \node  (6) at (1.5, 1.5) {};
    \node  (7) at (2, 2.5) {};
    \node  (8) at (2.5, 2.5) {};
    \node  (9) at (2.875, 1.75) {};
    \node  (10) at (3.25, 1) {};
    \node  (11) at (4, 1) {};
    \node  (12) at (5, 3) {};
    \node  (13) at (5.5, 2) {};
    \node  (14) at (6, 2) {};
    \node  (15) at (6.75, 0.5) {};
    \node  (16) at (4.8, 2.6) {};
    \node  (17) at (6.75, -0.25) {$t$};
    \node  (18) at (-0.25, 3.25) {$R_t$};
    \draw [->](0.center) to (2.center);
    \draw [->](0.center) to (1.center);
    \draw [thick, color=blue] (0.center) to (3.center);
    \draw [thick, color=blue, dashed] (3.center) to (4.center);
    \draw [thick, dashed](4.center) to (5.center);
    \draw [thick](5.center) to (6.center);
    \draw [thick, color=blue, dashed](6.center) to (7.center);
    \draw [thick](7.center) to (8.center);
    \draw [thick, color=red](8.center) to (9.center);
    \draw [thick, color=red, dashed](9.center) to (10.center);
    \draw [thick, dashed](10.center) to (11.center);
        \draw [thick, color=blue](11.center) to (16.center);
    \draw [thick, color=blue, dashed](16.center) to (12.center);
    \draw [thick, color=red, dashed](12.center) to (13.center);
    \draw [thick](13.center) to (14.center);
    \draw [thick, color=red, dashed](14.center) to (15.center);
\end{tikzpicture}
\caption{A sample path of the process $\{R_t\}_{t\ge 0}$ whose orbit process $\{\bm{A}_t\}_{t\ge 0}$ has state space $\mathfrak{Z}$, with $\mathfrak{Z}^+=\typeblue{\mathfrak{Z}^+_{\mbox{solid}}}\cup\typeblue{\mathfrak{Z}^+_{\mbox{dashed}}}$, $\mathfrak{Z}^-=\typered{\mathfrak{Z}^-_{\mbox{solid}}}\cup\typered{\mathfrak{Z}^-_{\mbox{dashed}}}$ and $\mathfrak{Z}^0={\mathfrak{Z}^0_{\mbox{solid}}}\cup{\mathfrak{Z}^0_{\mbox{dashed}}}$.}
\label{fig:multijumpFRAP}
\end{figure}

\begin{assumption} 
	\label{assump:pm1} 
Henceforth, $\nu^+_i=1$ for  $i\in\{1,\dots,n^+\}$ and $\nu^-_j=-1$ for  $j\in\{1,\dots,n^-\}$. 
\end{assumption} 

Under this assumption, $\{R_t\}_{t\ge 0}$ reduces to
\begin{align}
	\label{eq:p1m1additive1}
	R_t=\int_0^t \Indic{\bm{A}_s\in\mathfrak{Z}^+} -\Indic{\bm{A}_s\in\mathfrak{Z}^-} \dd s,\quad t\ge 0.
	\end{align}
This greatly simplifies the analysis of the RAP-modulated fluid process; results from the general case can be recovered through standard time-change techniques \cite{Rogers:1994uoa}.


As in Section~\ref{sec:proto-FRAP}, in order to avoid trivial paths we impose the following. 

\begin{Condition} 
For any $k\in\{+,-\}$,
\begin{align}
	\label{eq:jumpshappen1}
\lim_{t\rightarrow\infty}\mathds{P}\left(\left.\bm{A}_s\in \mathfrak{Z}^k\cup\mathfrak{Z}^0 \;\forall s\in[0,t]\;\right| \bm{A}_0=\bm{\alpha}\right) = 0 \quad \mbox{for all } \bm{\alpha}  \in\mathfrak{Z}^k\cup\mathfrak{Z}^0.
\end{align}
\end{Condition} 
%

\subsection{Preliminaries}
{Denote by $\mathds{P}_{\bm{\alpha}}$ the probability law of $\mathcal{X}=\{(R_t,\bm{A}_t)\}_{t\ge 0}$ conditioned on the event $\{\bm{A}_0=\bm{\alpha}\}$, and denote by $\mathds{E}_{\bm{\alpha}}$ its associated expectation. Let $(\Omega^*,\mathcal{F},\{\mathcal{F}_t\}_{t\ge 0}, \{\mathds{P}_{\bm{\alpha}}\}_{\bm{\alpha}\in \mathfrak{Z}})$ be the canonical probability space associated to the Markov process $\{(R_t,\bm{A}_t)\}_{t\ge 0}$, and w.l.o.g. assume that $\{(R_t,\bm{A}_t)\}_{t\ge 0}$ is defined there. More specifically, for $\omega=\{(r_t,\bm{a}_t)\}_{t\ge 0}\in\Omega^*$,
\begin{equation}\label{eq:shiftX1} R_t(\omega)=r_t\quad\mbox{and} \quad \bm{A}_t(\omega)=\bm{a}_t\qquad\mbox{for all}\quad t\ge 0.\end{equation}
Heuristically speaking, each $\{(r_t,\bm{a}_t)\}_{t\ge 0}\in\Omega^*$ corresponds to a feasible \emph{path} of $\{(R_t,\bm{A}_t)\}_{t\ge 0}$.}

%
{Let $\Omega$ be the set of paths $\{(r_t,\bm{a}_t)\}_{t\ge 0}\in\Omega^*$ such that:}
\begin{itemize}
  \item $\{\bm{a}_t\}_{t\ge 0}$ has a finite number of jumps on each compact time interval,
  \item for all $s\ge 0$, neither $\bm{a}_t\in\mathfrak{Z}^+\cup \mathfrak{Z}^0$ for all $t\ge s$, nor $\bm{a}_t\in\mathfrak{Z}^-\cup \mathfrak{Z}^0$ for all $t\ge s$,
  \item there are no $s, t\ge 0$, $s\neq t$ such that $r_s=r_t$, $\bm{a}_{s^-}\in\mathfrak{Z}^{k}$, $\bm{a}_{s}\neq \bm{a}_{s^-}$, $\bm{a}_{t^-}\in\mathfrak{Z}^{\ell}$, $\bm{a}_{t}\neq \bm{a}_{t^-}$, for $k, \ell\in\{+,-\}$; in other words, no two jumps of $\{\bm{a}_t\}_{t\ge 0}$ happen while at the same level of $\{r_t\}_{t\ge 0}$.
\end{itemize}
\begin{remark} 
The paths in $\Omega$ are those that are regarded as \emph{nice}, and are the ones that will be considered throughout all the arguments in this paper. Note that the event $\{\mathcal{X}\in\Omega^*\setminus\Omega\}$ is $\mathds{P}_{\bm{\alpha}}$-null for all $\bm{\alpha}\in\mathfrak{Z}$; since the $\sigma$-algebra $\mathcal{F}$ is assumed to be complete, then the event $\{\mathcal{X}\in\Omega\}$ is an element of $\mathcal{F}$. For this reason, from here on we restrict the sample space of $\mathcal{X}$ to $\Omega$ and, with a slight abuse of notation, refer to its probability measure $\mathds{P}_{\bm{\alpha}}|_{\{\mathcal{X}\in\Omega\}}$ as $\mathds{P}_{\bm{\alpha}}$.
\end{remark} 

For any $s\ge 0$, define the \emph{shift operator} {$\theta_s:\Omega \mapsto \Omega$}
by
\begin{align*} 
	\theta_s\{(r_t,\bm{a}_t)\}_{t\ge 0} := \{(r_{t+s}-r_s, \bm{a}_{t+s})\}_{t\ge 0}.
\end{align*} 
%
{In particular, according to (\ref{eq:shiftX1})
\[R_t\circ\theta_s=R_{t+s}-R_s\quad\mbox{and}\quad\bm{A}_t\circ\theta_s=\bm{A}_{t+s}\qquad\mbox{for all}\quad t\ge 0,s\ge 0.\]}
%
%

Now, let $Z$ be an $\mathcal{F}_t$-stopping time. By the strong Markov property of PDMPs \cite{Davis:1984vi}, it follows that $\mathcal{X}$ also has the strong Markov property. This implies that for any $\mathcal{F}$-measurable bounded function $f:\Omega\mapsto \mathds{R}^k$, for $k\ge 1$, we have 
\[\mathds{E}_{\bm{\alpha}}\left[\left. f(\mathcal{X}\circ\theta_Z)\Indic{Z<\infty} \;\right| \mathcal{F}_Z\right] = \mathds{E}_{\bm{A}_{Z}}\left[f(\mathcal{X})\right]\Indic{Z<\infty},\]
{where $\mathcal{X}\circ\theta_Z(\omega)=\mathcal{X}(\theta_{Z(\omega)}(\omega))$}.
Next we investigate a useful implication of the minimality property of $\orbit$ and Lemma \ref{lem:bminimality}.
\begin{Corollary}
	\label{cor:bminimality}
Fix $x\ge 0$, $k\in\mathcal{S}$ and let $\{\Pi(s)\}_{s\ge 0}$ be a collection of matrices with identical dimensions. If $\{\bm{\alpha}\Pi(s): \bm{\alpha}\in\mathfrak{Z}^k, s\in [0, x]\}$ is uniformly entrywise-bounded, then so is $\{\Pi(s):s\in [0, x]\}$.
\end{Corollary}
\begin{proof}
Suppose $\exists \{s_n\}_{n=1}^\infty\subset [0,x]$ and $j, \ell \in \{1,\dots , \eta^k\}$ such that  $\{\Pi(s_n)_{j\ell}\}_{n=1}^\infty$ is unbounded. Let $\{\bm{g}_m\}_{m}\subset \mathfrak{Z}^k$ be a collection of $\eta^k$ linearly independent vectors and let $\{c_m\}_m\subset\mathds{R}$ be such that $\sum_m c_m \bm{g}_m = \bm{e}_j'$, where $\bm{e}_j$ denotes the unit column vector that is nonzero only at its $j$th entry. Then 
\[\left\{\sum_m c_m (\bm{g}_m\Pi(s_n))_\ell \right\}_{n=1}^\infty = \left\{\Pi(s_n)_{j\ell}\right\}_{n=1}^\infty\]
is unbounded, which is a contradiction.
\end{proof}
Next, we analyze the solution of a particular integral matrix-equation which is key to our analysis throughout this paper.
\begin{theorem}	
	\label{th:explicitphi}
Let $\Pi:\mathds{R}_+ \mapsto \mathds{R}^{m \times m}$ and $A, B \in \mathds{R}^{m \times m}$ be such that
   $\Pi(\cdot)$ is uniformly entrywise-bounded in compact intervals, and 
  for all $x\ge 0$
\begin{align}\label{eq:explicitphi}
\Pi(x)= e^{A x} + \int_0^x e^{As}B\Pi(x-s)\dd s.
\end{align}
%
Then, $\Pi(x)=e^{(A+B)x}$ for all $x\ge 0$.
\end{theorem}
\begin{proof} 
The assumptions of $\Pi(\cdot)$ imply that it is infinitely differentiable in $(0,\infty)$. 
Premultiplying (\ref{eq:explicitphi}) by $e^{-A x}$ and differentiating gives us 
\begin{align*}
-e^{-A x}A\Pi(x) + e^{-A x}\Pi'(x)= e^{-A x}B\Pi(x),
\end{align*} 
so that $\Pi'(x) = (A + B)\Pi(x),$ 
with initial condition $\Pi(0) = I$. Thus the result follows.
\end{proof}

\subsection{Properties of the orbit process}

In the following, we study the average behaviour of the process $\{\bm{A}_t\}_{t\ge 0}$ over time restricted to certain events. By (\ref{eq:solODE1}), we have that if $\orbit$ is simple and does not leave $\mathfrak{Z}^k$ by time~$t$, for $k \in \{+, -\}$, then the row vector $\bm{A}_t$ is proportional to $\bm{A}_0e^{C^k t}$. The following is an analogous result for the general orbit process.
\begin{theorem}
	\label{th:expandedB1}
	For all $k\in\{+,-,0\}$ and $\bm{\alpha}\in\mathfrak{Z}^k$,
\begin{align}
	\label{eq:meanCk}
\mathds{E}_{\bm{\alpha}}\left[\bm{A}_t\Indic{\bm{A}_s\in \mathfrak{Z}^k\;\forall s\in[0,t]}\right] = \bm{\alpha} e^{C^k t}\quad \mbox{ for } t\ge 0,
\end{align} 
where $C^k\in\mathds{R}^{\eta^k\times{\eta^k}}$ is of the form
%
\[
	C^k=\begin{pmatrix}
	C^k_{11}& \cdots& C^k_{1n^k}\\
	\vdots&\ddots&\vdots\\
	C^k_{n^k1}& \cdots& C^k_{n^kn^k}
	\end{pmatrix}.
\]
\end{theorem}

\begin{proof}
 Let $s_0(\bm{\mathcal{A}}):= \inf\{s>0: \bm{A}_{s^-}\neq \bm{A}_s\}$ be the time of the first jump of $\bm{\mathcal{A}}$, and for $t\ge 0$ let
%
\begin{align} 
	L_t:=\#\{s\in (0,t]:  \bm{A}_{s^-}\neq \bm{A}_s\}
\end{align}
be the number of jumps of $\bm{\mathcal{A}}$ in the interval $(0,t]$. When there is no ambiguity, we drop the dependency on $\bm{\mathcal{A}}$ and write simply $L_t$.
 
First, we show, by induction, that  for all $n\ge 0$,
\begin{align}
	\label{eq:multisigma1}
\mathds{E}_{\bm{\alpha}}\left[\bm{A}_t\Indic{\bm{A}_s\in\mathfrak{Z}^k\;\forall s\in[0,t], \;L_t=n}\right] = \bm{\alpha}\Sigma_n(t)
\end{align}
for unique continuous matrices $\{\Sigma_n (\cdot)\}_{n\ge 0}$. 
  
  {\bf Case $n=0$.} Equation (\ref{eq:genorbitnojumps1}) implies that
  \begin{align}
  \mathds{E}_{\bm{\alpha}}&\left[\bm{A}_t\Indic{\bm{A}_s\in\mathfrak{Z}^k\;\forall s\in[0,t],\; L_t=0}\right] = \bm{\alpha}e^{\Gamma^k t}.\label{eq:auxbase1}
\end{align}
Thus, (\ref{eq:multisigma1}) is true if we choose $\Sigma_0(t) = e^{\Gamma^k t}$, and this solution is unique by Lemma \ref{lem:bminimality}.

{\bf Inductive part.} Suppose that (\ref{eq:multisigma1}) is true for some $n\ge 1$. Then, for $r\in(0,t)$
 \begin{align}
\mathds{E}_{\bm{\alpha}}&\left[\bm{A}_t\Indic{\bm{A}_s\in\mathfrak{Z}^k\;\forall s\in[0,t], \;L_t=n+1,\; s_0\in [r-\dd r, r]}\right] \nonumber\\
  & = \mathds{E}_{\bm{\alpha}}\left[\mathds{E}_{\bm{\alpha}}\left[\left.\bm{A}_t\Indic{\bm{A}_s\in\mathfrak{Z}^k\;\forall s\in[r,t],\; L_t\circ\theta_r=n}\;\right| \mathcal{F}_r\right]\Indic{s_0\in [r-\dd r, r], \;\bm{A}_{r}\in\mathfrak{Z}^k}\right] \nonumber \\
  & =  \mathds{E}_{\bm{\alpha}}\left[\mathds{E}_{\bm{A}_{r}}\left[\bm{A}_{t-r}\Indic{\bm{A}_s \in\mathfrak{Z}^k\;\forall s\in[0,t-r],\; L_{t-r} =n}\right]\Indic{s_0\in [r-\dd r, r], \;\bm{A}_{r}\in\mathfrak{Z}^k}\right] \nonumber\\
 & =  \mathds{E}_{\bm{\alpha}}\left[\left(\bm{A}_r\Sigma_n(t-r)\right)\Indic{s_0\in [r-\dd r,r], \;\bm{A}_{r}\in\mathfrak{Z}^k}\right] \nonumber\\
  & =  \mathds{E}_{\bm{\alpha}}\left[\bm{A}_r\Indic{s_0\in [r-\dd r,r], \;\bm{A}_{r}\in\mathfrak{Z}^k}\right]\Sigma_n(t-r), \nonumber
\end{align}
where the strong Markov property is used in the second equality, and the induction hypothesis in the third. Now, for any $\bm{\alpha}^*\in\mathfrak{Z}^k$,
\begin{align}
\mathds{E}_{\bm{\alpha}^*}\left[\bm{A}_{\dd r}\Indic{s_0\in [0,\dd r], \;\bm{A}_{\dd r}\in\mathfrak{Z}^k}\right] &= \sum_{1\le i\le n^k}\sum_{\substack{1\le j\le n^k, \\ j\neq i}} \frac{\bm{\alpha}^* \widehat{C}^k_{ij}}{\bm{\alpha}^* \widehat{C}^k_{ij}\bm{1}} \left(\bm{\alpha}^* \widehat{C}^k_{ij}\bm{1}\dd r\right) \Indic{\bm{\alpha}^*\in\mathfrak{Z}^k_i} \nonumber\\
& = \bm{\alpha}^*\left(C^k - \Gamma^k\right)\dd r, \nonumber
\end{align}

Thus, 
\begin{align*}
& \mathds{E}_{\bm{\alpha}} \left[\bm{A}_r\Indic{s_0\in (r-\dd r,r),\; \bm{A}_{r}\in\mathfrak{Z}^k}\right] \\
& = \mathds{E}_{\bm{\alpha}} \left[\mathds{E}_{\bm{\alpha}}\left[\left.\bm{A}_r\Indic{s_0\in (r-\dd r, r),\; \bm{A}_{r}\in\mathfrak{Z}^k}\;\right| \mathcal{F}_{r-\dd r}\right]\Indic{L_{r-\dd r} =0}\right]\\
& = \mathds{E}_{\bm{\alpha}} \left[\mathds{E}_{\bm{A}_{r-\dd r}}\left[\bm{A}_{\dd r}\Indic{s_0\in (0,\dd r),\;\bm{A}_{\dd r}\in\mathfrak{Z}^k}\right]\Indic{L_{r-\dd r} =0}\right]\\
& = \mathds{E}_{\bm{\alpha}} \left[\left(\bm{A}_{r-\dd r}\left({C}^k - \Gamma^k\right)\right)\dd r\Indic{L_{r-\dd r} =0}\right]\\
& = \bm{\alpha}e^{\Gamma^k (r-\dd r)}\left({C}^k- \Gamma^k\right)\dd r = \bm{\alpha}e^{\Gamma^k r}\left({C}^k - \Gamma^k\right)\dd r,
\end{align*}
where (\ref{eq:auxbase1}) is used in the second last equality.
Thus,
\begin{align*}
  \mathds{E}_{\bm{\alpha}}&\left[\bm{A}_t\Indic{\bm{A}_s\in \mathfrak{Z}^k\;\forall s\in[0,t],\; L_t=n+1}\right] = \bm{\alpha}\Sigma_{n+1}(t)
  \end{align*}
  with
  \begin{align}
  	\label{eq:Sigmanint1}
  \Sigma_{n+1}(t) = \int_0^t e^{\Gamma^k r}\left({C}^k - \Gamma^k\right)\Sigma_n(t-r)\dd r,\end{align}
  which is continuous. The uniqueness of $\Sigma_{n+1}(\cdot)$ is guaranteed by Lemma \ref{lem:bminimality}.
Now, 
\begin{align}
	\label{eq:sumSigmanm}
	\mathds{E}_{\bm{\alpha}}\left[\bm{A}_t\Indic{\bm{A}_s\in \mathfrak{Z}^k\;\forall s\in[0,t],\; L_t\le m }\right]= \bm{\alpha} \sum_{n=0}^m \Sigma_n(t)\quad\mbox{for all } m\ge 0.
	\end{align}
The boundedness of $\mathfrak{Z}^k$ and Corollary \ref{cor:bminimality} imply that
$\{\sum_{n=0}^m \Sigma_n(s): m\ge 0, s \in [0,t]\}$ is entrywise-bounded.
%
%
Moreover, since 
\begin{align}
	\label{eq:simsigmatoinfty1}
\mathds{E}_{\bm{\alpha}}\left[\bm{A}_t\Indic{\bm{A}_s\in\mathfrak{Z}^k\;\forall s\in[0,t],\; L_t\ge  m }\right]= \bm{\alpha} \sum_{n=m}^\infty \Sigma_n(t)\end{align} 
which converges to $\bm{0}$ as $m\rightarrow 0$, Lemma \ref{lem:bminimality} implies that $\sum_{n=m}^\infty \Sigma_n(t)\rightarrow\bm{0}$ as $m\rightarrow \infty$. This guarantees the existence of $\Sigma_n(\cdot):=\sum_{n=1}^\infty \Sigma_n(\cdot)$ which is uniformly entrywise-bounded in compact intervals. Hence,
\[
	\mathds{E}_{\bm{\alpha}}\left[\bm{A}_t\Indic{\bm{A}_s\in \mathfrak{Z}^k \;\forall s\in[0,t]}\right]= \bm{\alpha}  \Sigma(t)\quad \mbox{ for } t\ge 0.
\]
Using the entrywise-boundedness of $\{\sum_{n=0}^m \Sigma_n(t): m\ge 0, s\in [0,t]\}$, the Bounded Convergence Theorem and (\ref{eq:Sigmanint1}), we obtain
\begin{align}
  \label{eq:integralSigma1}
  \Sigma(t) = e^{\Gamma^k t} + \int_0^t  e^{\Gamma^k s} \left({C}^k - \Gamma^k\right)\Sigma(t-s)\dd s\qquad (t\ge 0).
  \end{align}
%

Theorem \ref{th:explicitphi} and (\ref{eq:integralSigma1}) imply $\Sigma(t)=\exp \left(\left[\Gamma^k + ({C}^k - \Gamma^k)\right] t \right)=e^{{C}^kt}$, so (\ref{eq:meanCk}) follows.
\end{proof}  
%
\begin{remark}
In this paper, the majority of the proofs have a structure similar to that of Theorem \ref{th:expandedB1}: first an induction, then uniqueness and boundedness arguments, and finally finding the solution to a certain integral equation. For the sake of brevity, from here on we omit the steps concerning uniqueness and boundedness of matrices, given that they can be deduced from arguments analogous to those in the proof of Theorem \ref{th:expandedB1}.
\end{remark}

Note that since $\bm{A}_t\bm{1}=1$,  
\[\mathds{P}_{\bm{\alpha}}\left(\bm{A}_s\in\mathfrak{Z}^k\;\forall s\in[0,t]\right) = \bm{\alpha}e^{C^k t}\bm{1}.\]

While the matrices $\widehat{C}_{ij}^k$ have a deterministic physical meaning for $\boldsymbol{\mathcal{A}}$ given by (\ref{eq:ODEFRAP1}), $C^k$ does not. Nevertheless, Theorem \ref{th:expandedB1} implies that $C^k$ does have a role in the behaviour of $\boldsymbol{\mathcal{A}}$, not in a deterministic sense, but in an average sense instead. Lemma~\ref{lem:basicpropFRAP1} elaborates on this further. 

For $k\in\{+,-,0\}$, define
\begin{align}
	\label{def:rho}  
	\rho_k := \inf\{s\ge 0: \bm{A}_s\notin\mathfrak{Z}^k\},
\end{align} 
 the \emph{first exit time} of $\bm{\mathcal{A}}$ from $\mathfrak{Z}^k$. When there is no ambiguity, we drop the dependency on $\bm{\mathcal{A}}$ and write $\rho_k$.

\begin{Lemma}
	\label{lem:basicpropFRAP1}
	%
For $\ell, k \in \{+, - ,0\}$, $\ell\neq k$, and $t\ge 0$,
\begin{align}
	\label{eq:auxdensity1} 
	\mathds{E}_{\bm{\alpha}}\left[\bm{A}_{\rho_k}\Indic{\rho_k\in [t, t+ \dd t], \;\bm{A}_{\rho_k}\in\mathfrak{Z}^\ell}\right] & = \bm{\alpha} e^{C^k t}D^{k\ell}\dd t,  \\ 
%
	\label{eq:meanafterexiting1}
	\mathds{E}_{\bm{\alpha}}\left[\bm{A}_{\rho_k}\Indic{\bm{A}_{\rho_k}\in\mathfrak{Z}^\ell}\right] & = \bm{\alpha} (-C^k)^{-1}D^{k\ell},
\end{align}
where $\bm{\alpha}\in\mathfrak{Z}^k$ and $D^{k\ell}\in\mathds{R}^{\eta^k\times{\eta^\ell}}$ is of the form
%
\[
	D^{k\ell}=\begin{pmatrix}
	D^{k\ell}_{11}& \cdots& D^{k\ell}_{1n^\ell}\\ 
	\vdots&\ddots&\vdots\\
	D^{k\ell}_{n^k1}& \cdots& D^{k\ell}_{n^kn^\ell}\end{pmatrix}.
\]
\end{Lemma}
\begin{proof}
First, note that
\begin{align*}
\mathds{E}_{\bm{\alpha}}&\left[\bm{A}_{\rho_k}\Indic{\rho_k\in [0, \dd t], \; \bm{A}_{\rho_k}\in \mathfrak{Z}^\ell }\right]  = \sum_{1\le i\le n^k}\sum_{1\le j\le n^\ell} \frac{\bm{\alpha} \widehat{D}^{k\ell}_{ij}}{\bm{\alpha} \widehat{D}^{k\ell}_{ij}\bm{1}}  \left(\bm{\alpha} \widehat{D}^{k\ell}_{ij}\bm{1}\dd t\right) \Indic{\bm{\alpha}\in\mathfrak{Z}^k_i} =  \bm{\alpha}D^{k\ell}\dd t.
\end{align*}

Thus, 
\begin{align*}
& \mathds{E}_{\bm{\alpha}}\left[\bm{A}_{\rho_k}\Indic{\rho_k\in [t, t+ \dd t], \; \bm{A}_{\rho_k}\in \mathfrak{Z}^\ell }\right] \\
& = \mathds{E}_{\bm{\alpha}}\left[\mathds{E}_{\bm{\alpha}}\left[\left.\bm{A}_{\rho_k}\Indic{\rho_k\in [t, t+ \dd t],\; \bm{A}_{\rho_k}\in \mathfrak{Z}^\ell}\;\right| \mathcal{F}_t\right]\Indic{\bm{A}_s\in\mathfrak{Z}^k\;\forall s\in[0,t]}\right]\\
& = \mathds{E}_{\bm{\alpha}}\left[\mathds{E}_{\bm{A}_t}\left[\bm{A}_{\rho_k}\Indic{\rho_k\in [0, \dd t], \;\bm{A}_{\rho_k}\in \mathfrak{Z}^\ell}\right]\Indic{\bm{A}_s\in\mathfrak{Z}^k\;\forall s\in[0,t]}\right]\\
& = \mathds{E}_{\bm{\alpha}}\left[\left(\bm{A}_tD^{k\ell}\dd t\right)\Indic{\bm{A}_s\in\mathfrak{Z}^k\;\forall s\in[0,t]}\right]\\
&  = \bm{\alpha} e^{C^k t}D^{k\ell}\dd t.
\end{align*}
By Theorem \ref{th:expandedB1}  and Equation \eqref{eq:jumpshappen1}, we have $\lim_{t\rightarrow\infty}\bm{\alpha}e^{C^k t} = \bm{0}$. Since this holds for all $\bm{\alpha}\in\mathfrak{Z}^k$, Lemma \ref{lem:bminimality} implies that $\lim_{t\rightarrow\infty}e^{C^k t} = 0$ and so the eigenvalue of maximum real part of $C^k$ must have a strictly negative real part. Then
\begin{align*}
\mathds{E}_{\bm{\alpha}}\left[\bm{A}_{\rho_k}\Indic{\bm{A}_{\rho_k}\in\mathfrak{Z}^\ell }\right] & = \int_0^\infty \bm{\alpha} e^{C^k t}D^{k\ell}\dd t  =  \bm{\alpha} (-C^k)^{-1}D^{k\ell}.
\end{align*}
\end{proof}
Once again, while the matrices $\{\widehat{D}^{k\ell}_{ij}\}$ dictate in a deterministic way where $\orbit$ lands after a jump, $D^{k\ell}$ describes such landings in an \emph{average} sense.
 
\section{Case $n^0=0$} 
	\label{sec:nozerorates}
	Here, we focus on the case $n^0=0$, that is, the state space of $\orbit$ is simply $\mathfrak{Z}^+\cup\mathfrak{Z}^-$ and $\mathcal{R}$ is nowhere piecewise constant. In such a case, our analysis heavily relies on analysing the \emph{up-down and down-up peaks} of $\mathcal{R}$, that is, the epochs at which $\orbit$ jumps from $\mathfrak{Z}^+$ to $\mathfrak{Z}^-$ and from $\mathfrak{Z}^-$ to $\mathfrak{Z}^+$, respectively.
	
\subsection{First-return probabilities}
	\label{sec:firstzempty}
In this section we study the event in which the level process downcrosses its initial point. More specifically, 
define
\begin{align*}
&\tau_-:=\inf\left\{t>0: R_t \le 0\right\}, \quad \tau_+ :=\inf\left\{t>0: R_t \ge 0\right\}, \\
&\Omega_-:=\{\omega \in \Omega: \bm{A}_0 \in \mathfrak{Z}^+, \tau_- < \infty\}, \quad \Omega_+:=\{\bm{A}_0 \in \mathfrak{Z}^-, \tau_+ < \infty \}.
\end{align*}
%
The random variable $\tau_-$ ($\tau_+$) corresponds to the first time the process $\mathcal{R}$ visits $[0,\infty)$ ($(-\infty,0]$), and $\Omega_-$ ($\Omega_+$) is the set of paths in $\Omega$ whose orbit starts in $\mathfrak{Z}^+$ ($\mathfrak{Z}^-$) and whose level eventually downcrosses (upcrosses) level $0$ in finite time. 

We are interested in computing
$\mathds{E}_{\bm{\alpha}}\left[\bm{A}_{\tau_-}\Indic{\tau_- <\infty}\right]$  for $\bm{\alpha}\in\mathfrak{Z}^+.$ 
To that end, we borrow ideas from the FP3 algorithm of \cite{guo2001nonsymmetric}, whose probabilistic interpretation was developed in \cite{bean2005algorithms} and requires a particular partition of all paths in $\Omega_-$. Recall that we assume $\{R_t\}_{t\ge 0}$ is piecewise linear with slopes $\pm1$, which implies that for all $s,t\ge 0$ and $k\in\{+,-\}$,
\begin{align}\label{eq:leveltime1}
\big\{\bm{A}_t\in\mathfrak{Z}^k\big\}\cap\big\{|R_{t+s}-R_t|=s\big\} = \left\{\bm{A}_u\in\mathfrak{Z}^k\;\forall u\in[t,t+s)\right\}.
\end{align}

Let  
%
\[\Omega_1 := \left\{\omega\in\Omega_-:  \{\bm{A}_s\}_{s\ge 0} \mbox{ has exactly one jump from } \mathfrak{Z}^+ \mbox{ to }  \mathfrak{Z}^- \mbox{ during } [0,\tau_-)\right\}, \] 
the set of sample paths whose level component returns to $0$ after a single interval of increase and a single interval of decrease; in other words, there is no down-up peak before $\tau_-$. Then, $\Omega_{-} \setminus \Omega_1$ is the set of sample paths whose level has one or more down-up peaks before $\tau_-$. Define 
%
\[
	p := \left\{\begin{array}{ccl}
	\inf\left\{y\ge 0: \mbox{$\exists t\in[0,\tau_-)$ such that $R_t=y$, $\bm{A}_{t^-}\in\mathfrak{Z}^-, \bm{A}_{t}\in\mathfrak{Z}^+$} \right\}&\mbox{if}& \mathcal{X} \in \Omega_-\setminus\Omega_1, \\ 
	\infty &\mbox{if}& \mathcal{X} \in (\Omega_-\setminus\Omega_1)^c.\end{array}\right.
	\]
	
	For every sample path in $\Omega_-\setminus\Omega_1$, $p$ corresponds to the lowest level at which there is a down-up peak before time $\tau_-$, and we decompose each path at the time such lowest level is attained, denoted by $T_2$. Let $T_1 := \inf \{t \geq 0: R_t = p\}$ be the first time the level reaches $p$ and $T_3 := \sup \{t \geq 0: R_t = p, t < \tau_{-}\}$ be the last time it does so before returning to level zero. Then, $T_1  = p$ by the assumption of $\pm 1$ rates, $T_2 = T_1 + \tau_{-}\circ\theta_{T_1}$, and $T_3 = T_2 + \tau_{-}\circ\theta_{T_2}$.

Now, we recursively define $\Omega_n'\subset\Omega_-$ and $\Omega_n\subset\Omega_-$. Let $\Omega_1' = \Omega_1$. For $n \geq 2$, a path $\omega \in \Omega_-$ is an element of $\Omega_n'$ if and only if $\theta_{T_1}\omega \in \Omega_{n - 1}' \cup \Omega_1$ and $\theta_{T_2}\omega \in \Omega_{n - 1}' \cup \Omega_1$. The collection $\{\Omega_n'\}_{n\ge 1}$ may be thought as a collection of paths of possibly increasing \emph{complexity}, in the sense that the \emph{excusion above level $p$ of a path in $\Omega_n'$} can be decomposed, at time $T_2$, into two consecutive excursions above $p$ (each increasing from $p$ and returning to $p$) of paths in $\Omega_{n-1}'\cup\Omega_1$. Figure \ref{fig:Omegan} exemplifies a path in $\Omega_n'$.

\begin{figure}[h]
	\begin{center} 
	\includegraphics[scale=0.44]{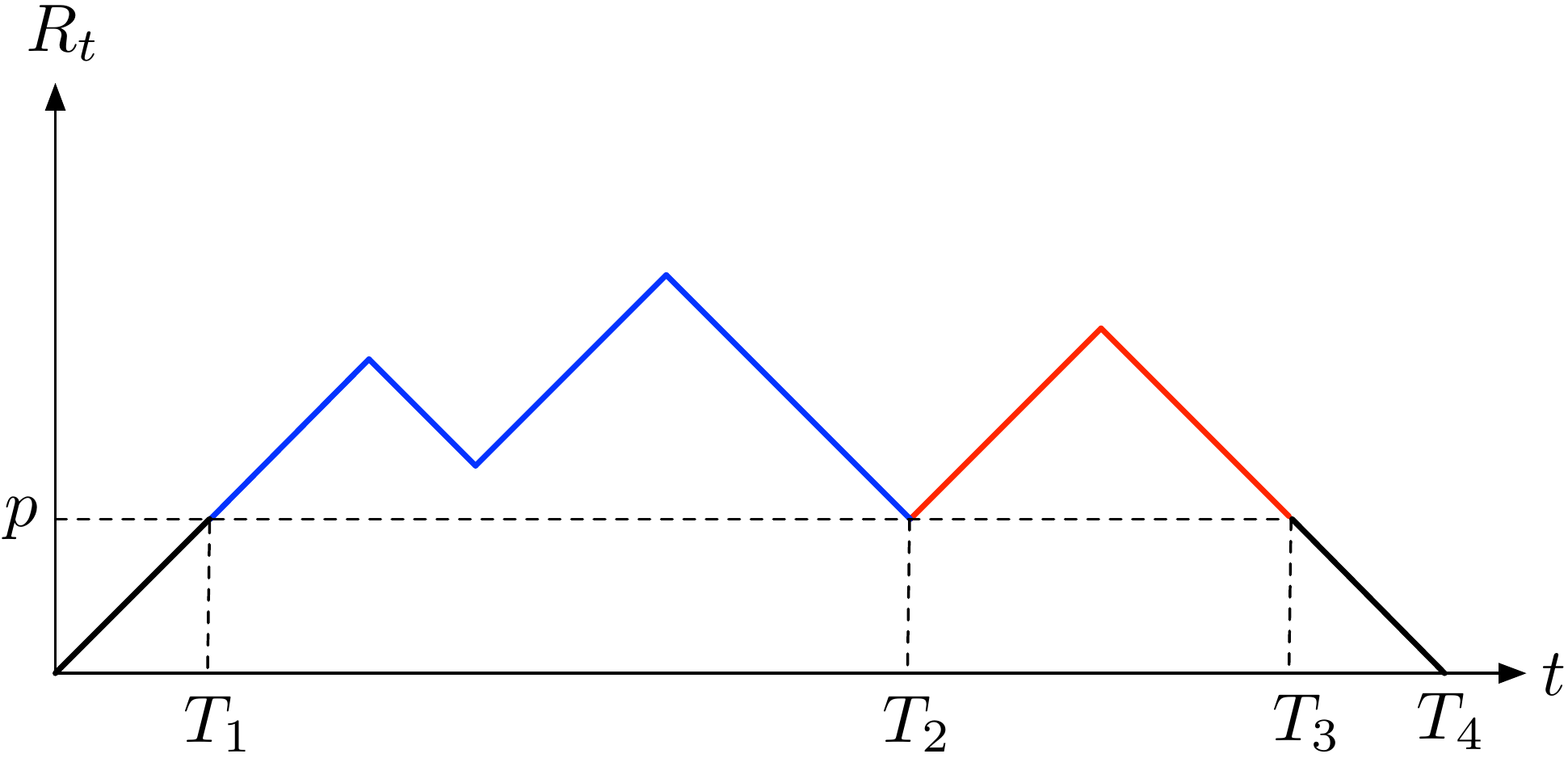}
	\label{fig:Omegan}
	\end{center} 
	\caption{An example of a level process corresponding to $\omega\in\Omega_3'$. The level process corresponding to $\theta_{T_1}\omega, \theta_{T_2}\omega\in\Omega_2'\cup\Omega_1$ up to their downcrossing of level $p$ are shown in blue and red, respectively. The blue and red segments, together, are also the sojourn above level $p$ of $\mathcal{R}$ prior to returning to level $0$.} 
  \label{fig:Omegan}
\end{figure}

Let $\Omega_n=\Omega_n'\cup\Omega_1$, then $\Omega_{n-1}\subset\Omega_{n}$ for $n\ge 2$ and $\Omega_-=\cup_{n=1}^\infty \Omega_n$. In Theorem~\ref{th:Psinintegral3}, we show that restricted to paths in $\Omega_n$, the mean value of $\bm{A}_{\tau_-}$ is characterized by a certain matrix $\Psi_n$. 
%

\begin{theorem}
	\label{th:Psinintegral3}
	For any given $\bm{\alpha}\in\mathfrak{Z}^+$ and for all $n\ge 1$,  
%
\begin{align}
	\label{eq:linpsin}
\mathds{E}_{\bm{\alpha}}\left[\bm{A}_{\tau_-}\Indic{\Omega_n}\right]=\bm{\alpha} \Psi_{n}, 
\end{align}
for unique matrices $\{\Psi_n\}_{n\ge 1}$ with
\begin{align}
	\Psi_{0}  & = \bm{0},  \nonumber \\ 
	\label{eq:psiintegral}
		\Psi_{n} & = \int_0^\infty e^{C^+y}\left(D^{+-} + \Psi_{n-1}D^{-+}\Psi_{n-1}\right)e^{C^-y} \dd y.
	\end{align}
\end{theorem}
\begin{proof} We prove by induction.

{\bf Case $n=1$.} Let $\rho_+$ be as in Lemma \ref{lem:basicpropFRAP1}, so that $\rho_+$ corresponds to the epoch at which the first transition from $\mathfrak{Z}^+$ to $\mathfrak{Z}^-$ occurs. Fix $y\ge 0$ and define the stopping times
  \begin{align*}
  Z_1 &:= y,\qquad
  Z_2 := Z_1 + \tau_-\circ\theta_{Z_1},\qquad
  Z_3 := Z_2 + y.
  \end{align*}

  Then, $\Omega_1\cap\left\{\rho_+\in (y, y +\dd y)\right\} = E_1\cap E_{2}\cap E_3$ a.s.,   where
  \begin{align*}
  E_1&=\{\bm{A}_s\in\mathfrak{Z}^+\mbox{ for all }s\in[0,Z_1)\},\\
  E_{2}&=\{\mbox{$\exists s \in (Z_1, Z_1 + \dd y)$ such that $\bm{A}_{s^-}\in\mathfrak{Z}^+$ and $\bm{A}_s\in\mathfrak{Z}^-$}\},\\
  E_3&=\{\bm{A}_s\in\mathfrak{Z}^-\mbox{ for all }s\in[Z_2,Z_3)\}.
  \end{align*}

  Moreover, on $\Omega_1\cap\{\rho_+\in (y, y +\dd y)\}$ we have that $Z_3=\tau_-$. Thus,
\begin{align*}
  \mathds{E}_{\bm{\alpha}}&\left[\bm{A}_{\tau_-}\Indic{\Omega_1\cap\{\rho_+\in (y, y +\dd y)\}}\right]\\
  & = \mathds{E}_{\bm{\alpha}}\left[\bm{A}_{Z_3}\Indic{E_1\cap E_{2}\cap E_3}\right] = \mathds{E}_{\bm{\alpha}}\left[\mathds{E}_{\bm{\alpha}}\left[\left.\bm{A}_{Z_3}\Indic{E_3}\;\right| \mathcal{F}_{Z_2}\right]\Indic{E_1\cap E_{2}}\right]\\
& = \mathds{E}_{\bm{\alpha}}\left[\left(\bm{A}_{Z_2}e^{C^- y}\right)\Indic{E_1\cap E_{2}}\right] = \mathds{E}_{\bm{\alpha}}\left[\mathds{E}_{\bm{\alpha}}\left[\left.\bm{A}_{Z_2}\Indic{E_{2}}\;\right| \mathcal{F}_{Z_1}\right]\Indic{E_1}\right]e^{C^- y}\\
& = \mathds{E}_{\bm{\alpha}}\left[\left(\bm{A}_{Z_1}D^{+-} \dd y\right)\Indic{E_1}\right]e^{C^- y} = \bm{\alpha}e^{C^+ y}D^{+-} e^{C^- y}\dd y,
  \end{align*}
  where the strong Markov property was used in the third and fifth equalities, Theorem \ref{th:expandedB1} was used in the third and last equalities, and (\ref{eq:auxdensity1}) in the fifth equality. Since this holds for all $\bm{\alpha}\in\mathfrak{Z}^+$, then Lemma \ref{lem:bminimality} implies that
  \[\Psi_1 = \int_{0}^\infty e^{C^+ y}D^{+-}e^{C^- y}\dd y\]
  is the only solution to (\ref{eq:linpsin}) for $n=1$.

{\bf Inductive part.} Suppose \eqref{eq:linpsin} is true for some $n\ge 1$. Fix $y > 0$. Define the stopping times
 \begin{align*}
  S_1&:= y,\quad
  S_2 := S_1 + \tau_-\circ\theta_{S_1},\quad
  S_3 := S_2 + \tau_+\circ\theta_{S_2},\quad
  S_4 := S_3 + \tau_-\circ\theta_{S_3},\quad
  S_5 := S_4 + y.
  \end{align*} 

 Note that
\begin{align}
	\label{eq:concatenationaux2}
	(\Omega_{n+1}\setminus \Omega_1)\cap\{p \in (y-\dd y, y)\}=B_1\cap B_2\cap B_{3}\cap B_4\cap B_5 \qquad\mbox{a.s.},
	\end{align}
where
\begin{align*}
B_1 & = \{\bm{A}_s\in\mathfrak{Z}^+\mbox{ for all }s \in [0,S_1)\},\\
B_2 & = \{\mathcal{X}\circ\theta_{S_1}\in\Omega_{n}\},\\
B_{3} & =\{\mbox{$\exists s\in (S_2, S_2 + \dd y)$ such that $\bm{A}_{s^-}\in\mathfrak{Z}^-$ and $\bm{A}_s\in\mathfrak{Z}^+$}\},\\
B_4 & = \{\mathcal{X}\circ\theta_{S_3}\in\Omega_{n}\},\\
B_5 & = \{\bm{A}_s\in\mathfrak{Z}^-\mbox{ for all }s \in [S_4,S_5)\}.
\end{align*}
Moreover, on $(\Omega_{n+1}\setminus \Omega_1)\cap\{p\in (y-\dd y, y)\}$ we have that $S_5=\tau_-$. Then,
\begin{align}
 \mathds{E}_{\bm{\alpha}}&\left[\bm{A}_{\tau_-}\Indic{(\Omega_{n+1}\setminus \Omega_1)\cap\{p\in (y-\dd y, y )\}}\right]\nonumber\\
& =  \mathds{E}_{\bm{\alpha}}\left[\bm{A}_{S_5}\Indic{\cap_{i=1}^5 B_i}\right] =  \mathds{E}_{\bm{\alpha}}\left[\mathds{E}_{\bm{\alpha}}\left[\left.\bm{A}_{S_5}\Indic{B_5}\;\right|\mathcal{F}_{S_4}\right]\Indic{\cap_{i=1}^4 B_i}\right]\nonumber\\
& = \mathds{E}_{\bm{\alpha}}\left[\left(\bm{A}_{S_4}e^{C^-y}\right)\Indic{\cap_{i=1}^4 B_i}\right] =  \mathds{E}_{\bm{\alpha}}\left[\mathds{E}_{\bm{\alpha}}\left[\left.\bm{A}_{S_4}\Indic{B_4}\;\right|\mathcal{F}_{S_3}\right]\Indic{\cap_{i=1}^3 B_i}\right]e^{C^-y}\nonumber\\
& = \mathds{E}_{\bm{\alpha}}\left[\left(\bm{A}_{S_3}\Psi_{n}\right)\Indic{\cap_{i=1}^3 B_i}\right]e^{C^-y} =  \mathds{E}_{\bm{\alpha}}\left[\mathds{E}_{\bm{\alpha}}\left[\left.\bm{A}_{S_3}\Indic{B_{3}}\;\right|\mathcal{F}_{S_2}\right]\Indic{\cap_{i=1}^2 B_i}\right]\Psi_{n}e^{C^-y}\nonumber\\
& = \mathds{E}_{\bm{\alpha}}\left[\left(\bm{A}_{S_2}D^{-+}\dd y\right)\Indic{\cap_{i=1}^2 B_i}\right]\Psi_{n}e^{C^-y} =  \mathds{E}_{\bm{\alpha}}\left[\mathds{E}_{\bm{\alpha}}\left[\left.\bm{A}_{S_2}\Indic{B_2}\;\right|\mathcal{F}_{S_1}\right]\Indic{B_1}\right]D^{-+}\Psi_{n}e^{C^-y}\dd y\nonumber\\
& = \mathds{E}_{\bm{\alpha}}\left[\left(\bm{A}_{S_1}\Psi_{n}\right)\Indic{B_1}\right]D^{-+}\Psi_{n}e^{C^-y}\dd y = \bm{\alpha}e^{C^+y}\Psi_{n}D^{-+}\Psi_{n}e^{C^-y}\dd y, \label{eq:Psiintaux5}
\end{align}
where the strong Markov property was used in the third, fifth, seventh and ninth equalities, Theorem \ref{th:expandedB1} in the third and last equalities, the induction hypothesis in the fifth and ninth equalities, and (\ref{eq:auxdensity1}) in the seventh equality. Thus,
\begin{align}
	\label{eq:psinaux4}
\mathds{E}_{\bm{\alpha}}&\left[\bm{A}_{\tau_-}\Indic{\Omega_{n+1}}\right] = \bm{\alpha}\int_0^\infty e^{C^+y}(D^{+-} + \Psi_{n}D^{-+}\Psi_{n}e^{C^-y}) \dd y = \bm{\alpha}\Psi_{n+1}.
\end{align}
As (\ref{eq:psinaux4}) holds for all $\bm{\alpha}\in\mathfrak{Z}$, $\Psi_{n+1}$ is uniquely determined by (\ref{eq:psiintegral}), which completes the proof.
\end{proof}
In Theorem \ref{th:Psinintegral3} we provided a recursion and an integral equation for $\Psi_n$. Corollary~\ref{cor:algorithmpsi} sets $\Psi_n$ as the solution of a Sylvester equation, and links $\{\Psi_n\}_{n\ge 1}$ to the probability of $\{R_t\}_{t\ge 0}$ ever returning to $0$. 

\begin{Corollary}
	\label{cor:algorithmpsi}
The matrices $\{\Psi_n\}_{n\ge 0}$ (with $\Psi_0:=0$) satisfy the recursive equation
\begin{align}
	\label{eq:FP3alg}
C^+\Psi_{n+1} + \Psi_{n+1}C^-= - D^{+-}-\Psi_{n}D^{-+}\Psi_{n},\quad n\ge 0.
\end{align}
%
Furthermore,
\begin{align}
	\mathds{E}_{\bm{\alpha}}\left[\bm{A}_{\tau_-}\Indic{\tau_-<\infty}\right] & = \bm{\alpha}\Psi, \label{eq:masterpsi3}\\
\mathds{P}_{\bm{\alpha}}(\tau_-< \infty) &= \bm{\alpha}\Psi\bm{1}.\label{eq:ruinprob1}
\end{align} 
where $\Psi := \lim_{n\rightarrow\infty}\Psi_n$. 
\end{Corollary}
\begin{proof} In the proof of (\ref{eq:meanafterexiting1}) we checked that the eigenvalues of maximal real part of $C^-$ and $C^+$ both have strictly negative real parts. Thus, premultiplying (\ref{eq:psiintegral}) by $C^+$ and integrating by parts we obtain
\begin{align*}
C^+\Psi_{n+1}  & = \int_0^\infty C^+ e^{C^+y}(D^{+-} + \Psi_{n}D^{-+}\Psi_{n})e^{C^-y} \dd y\\
& = \left[e^{C^+y}(D^{+-} + \Psi_{n}D^{-+}\Psi_{n})e^{C^-y}\right]_0^\infty  - \int_0^\infty e^{C^+y}(D^{+-} + \Psi_{n}D^{-+}\Psi_{n})e^{C^-y} C^-\dd y\\
& = [0 - (D^{+-} + \Psi_{n}D^{-+}\Psi_{n})] - \Psi_{n+1} C^-,\end{align*}
so that (\ref{eq:FP3alg}) follows. Using the Bounded Convergence Theorem and the fact that $\Omega_-=\cup_{n=1}^\infty \Omega_n$,
\begin{align}
	\label{eq:psiaux7}
		\mathds{E}_{\bm{\alpha}}\left[\bm{A}_{\tau_-}\Indic{\tau_-<\infty}\right] = \lim_{n\rightarrow\infty}\bm{\alpha}\Psi_n.
\end{align}
Since (\ref{eq:psiaux7}) holds for all $\bm{\alpha}\in\mathfrak{Z}^+$,  Lemma \ref{lem:bminimality} implies $\lim_{n\rightarrow\infty}\Psi_n =: \Psi$ exists and satisfies (\ref{eq:masterpsi3}). Equation (\ref{eq:ruinprob1}) follows by noticing that $\bm{A}_t\bm{1} = 1$ for all $t\ge 0$, a consequence of the affine nature of $\mathfrak{Z}^+$.
\end{proof}
The following is a property of the matrix $\Psi$ in the case $\mathds{P}(\Omega_-)=1$.
\begin{Proposition}
If $\mathds{P}_{\bm{\alpha}}(\tau_-<\infty)= 1$ for all $\bm{\alpha}\in\mathfrak{Z}^+$, then $\Psi\bm{1} = \bm{1}$.
\end{Proposition}
\begin{proof}
%
	
Since $\bm{\alpha}{\Psi}\bm{1}=1=\bm{\alpha}\bm{1}$ for all $\bm{\alpha}\in\mathfrak{Z}^+$, then Lemma \ref{lem:bminimality} implies $\Psi\bm{1} = \bm{1}$.
\end{proof}

\subsection{Downward record process}
	\label{sec:minimum}

For $x \geq 0$, define
%
%
$\tau^x_{-} :=\inf\{t > 0: R_t = -x, \bm{A}_t\in\mathfrak{Z}^-\},$
%
the first time the level process $\mathcal{R}$ downcrosses level $-x\le 0$. 
Define the \emph{downward record process} $\{(\ell_x, \bm{O}_x)\}_{x\ge 0}$ by
\begin{align*} 
	(\ell_x, \bm{O}_x) = \left\{\begin{array}{ccc}(R_{\tau_{-}^x},\;\bm{A}_{\tau_{-}^x})&\mbox{if}& \tau_{-}^x <\infty\\
	(\infty,\Delta)&\mbox{if}&\tau_{-}^x =\infty,
	\end{array}\right.
\end{align*} 
where $\Delta$ is some isolated cemetery state. If $\bm{O}_x\neq\Delta$ for $x\ge 0$, the vector $\bm{O}_x$ corresponds to the orbit value when the level process downcrosses $\ell_x=-x$ for the first time. We can see that $\{\bm{O}_x\}_{x\ge 0}$ is a (possibly absorbing) concatenation of orbits with state space $\mathfrak{Z}^-\cup\{\Delta\}$. In the following we compute the average value of $\bm{O}_x$ on the event that $\bm{O}_x\neq\Delta$.

\begin{theorem}
	\label{th:Gxbeta}
For all $\bm{\beta}\in\mathfrak{Z}^-$ and $x\ge 0$,
\begin{align}
	\mathds{E}_{\bm{\beta}}\left[\bm{O}_x\Indic{\tau_{-}^x <\infty}\right]=\bm{\beta}e^{(C^{-} +D^{-+}\Psi) x}.
\end{align}
\end{theorem}
\begin{proof}
Let $\kappa_0 := 0$, and for $n \geq 1$ define
\begin{align*} 
  \kappa_n := \inf\{t \geq \kappa_{n - 1}: R_t = \inf_{s \geq 0}R_s, \mbox{ and } \bm{A}_t \in \mathfrak{Z}^+\};
\end{align*} 
$\{\kappa_n\}_{n \geq 0}$ form the successive time epochs at which the infimum process $\{\inf_{s \geq 0} R_s\}_{s \geq 0}$ stops decreasing. For $n \geq 0$, let $-\sigma_n := R_{\kappa_n}$, the values of $\mathcal{R}$ at these epochs. Note that the sequence $\{\kappa_n + \tau_{-}\circ\theta_{\kappa_n}\}_{n \geq 0}$ form the successive time epochs at which the infimum process $\{\inf_{s \geq 0} R_s\}_{s \geq 0}$ starts decreasing, and $\{\kappa_{n + 1} - \kappa_n - \tau_{-}\circ\theta_{\kappa_n}\}_{n \geq 0}$ are the lengths of successive intervals in time at which $\mathcal{R}$ attains a new local minima. By Assumption~\ref{assump:pm1}, for $n \geq 0$
\begin{align*} 
	\sigma_n - \sigma_{n - 1} = \kappa_{n + 1} - \kappa_n - \tau_{-}\circ\theta_{\kappa_n}.
\end{align*}  
See Figure~\ref{fig:downwardFRAP} for an illustration.
 
%
\begin{figure}[h]
	\begin{center} 
	\includegraphics[scale=0.44]{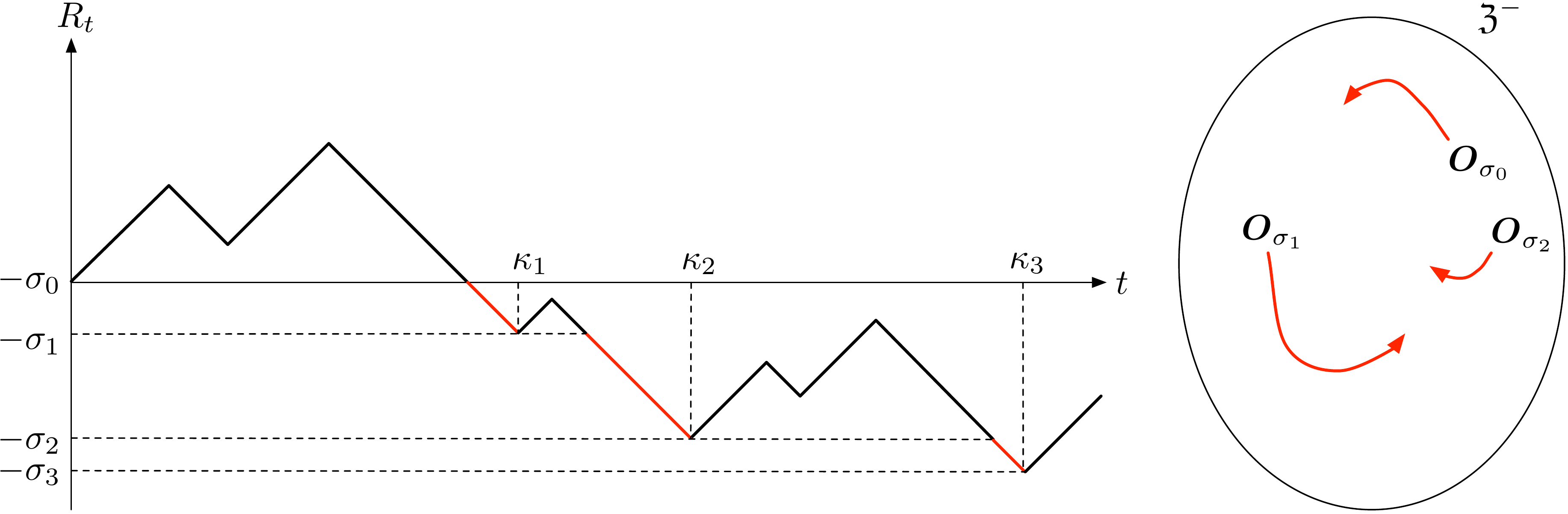}
	\end{center} 
    \caption{An example of the downward record process associated with $\mathcal{X}$. Left: Downward record levels are shown in red. Right: The concatenation $\{\bm{O}_{x}\}_{x\ge 0}$ of the corresponding orbit segments.}
    \label{fig:downwardFRAP}
    \end{figure}

For $x \geq 0$, let 
 \begin{align}
 	\label{eq:recorddown1} 
	%
	V_x :=\inf\{n\ge 1: \sigma_n > x\};
\end{align} 
	{we call $V_x$ the number of \emph{record downcrossings up to level $-x$}}. First, we prove by induction that for each $n\ge 1$, there exists a unique continuous matrix function $\Phi_n(\cdot)$ such that 
\begin{align}
	\label{eq:Phistarn}
	\mathds{E}_{\bm{\beta}}\left[\bm{A}_{\tau_{-}^x}\Indic{V_x=n}\right]=\bm{\beta}\Phi_n(x).
\end{align}

 {\bf Case $n=1$.} On $\{V_x=1\}$, \;$\tau_-^x=x$ by (\ref{eq:leveltime1}). Thus,
\begin{align*}
\mathds{E}_{\bm{\beta}}&\left[\bm{A}_{\tau_{-}^x}\Indic{V_x=1}\right] = \mathds{E}_{\bm{\beta}}\left[\bm{A}_x\Indic{\bm{A}_s\in\mathfrak{Z}^-}\;\forall s\in[0,x]\right] = \bm{\beta}e^{C^-x},
\end{align*}
so that (\ref{eq:Phistarn}) holds with $\Phi_1(x)=e^{C^- x}$. Uniqueness follows from Lemma \ref{lem:bminimality}.

 {\bf Inductive part.} Suppose (\ref{eq:Phistarn}) holds for some $n \; \ge 1$. Fix $y \in [0,x]$ and define the stopping times
 \begin{align*}
  S_1& := y,\quad
  S_2 := S_1 + \tau_+\circ \theta_{S_1},\quad
  S_3 := S_2 + \tau_-\circ\theta_{S_2},\quad \\
  S_4 & := S_3 + \inf\{t>0 : R_t\circ\theta_{S_3} = -(x-y)\}.
  \end{align*} 

Then, $\{V_x=n+1, \sigma_1 \in (y, y+\dd y )\}=B_1\cap B_{2}\cap B_3\cap B_4$ a.s., where
\begin{align*}
B_1 & = \{\bm{A}_s\in\mathfrak{Z}^-\mbox{ for all }s \in [0,S_1)\},\\
B_{2} & =\{\mbox{$\exists s\in (S_1, S_1 + \dd y)$ such that $\bm{A}_{s^-}\in\mathfrak{Z}^-$ and $\bm{A}_s\in\mathfrak{Z}^+$}\},\\
B_3 & = \{\mathcal{X}\circ\theta_{S_2}\in\Omega_-\},\\
B_4 & = \{V_{x-y}\circ\theta_{S_3}=n\}.
\end{align*}
Moreover, on $\{V_x=n+1, \sigma_1\in (y, y+\dd y )\}$ we have that $S_4=\tau_-^x$. Then,
\begin{align}
 \mathds{E}_{\bm{\beta}}&\left[\bm{A}_{\tau_{-}^x}\Indic{V_x=n+1,\; \sigma_1\in (-y - \dd y, -y )}\right]\nonumber\\
& =  \mathds{E}_{\bm{\beta}}\left[\bm{A}_{S_4}\Indic{\cap_{i=1}^4 B_i}\right] \nonumber \\
& =  \mathds{E}_{\bm{\beta}}\left[\mathds{E}_{\bm{\beta}}\left[\left.\bm{A}_{S_4}\Indic{B_4}\;\right|\mathcal{F}_{S_3}\right]\Indic{\cap_{i=1}^3 B_i}\right]\nonumber \\
& = \mathds{E}_{\bm{\beta}}\left[\left(\bm{A}_{S_3}\Phi_n(x-y)\right)\Indic{\cap_{i=1}^3 B_i}\right] \nonumber  \quad \mbox{(by the induction hypothesis)}  \\
& =  \mathds{E}_{\bm{\beta}}\left[\mathds{E}_{\bm{\beta}}\left[\left.\bm{A}_{S_3}\Indic{B_3}\;\right|\mathcal{F}_{S_2}\right]\Indic{\cap_{i=1}^2 B_i}\right]\Phi_n(x-y)\nonumber   \\
& = \mathds{E}_{\bm{\beta}}\left[\left(\bm{A}_{S_2}\Psi\right)\Indic{\cap_{i=1}^2 B_i}\right]\Phi_n(x-y) \nonumber \quad  \mbox{(by Corollary \ref{cor:algorithmpsi})} \\
& =  \mathds{E}_{\bm{\beta}}\left[\mathds{E}_{\bm{\beta}}\left[\left.\bm{A}_{S_2}\Indic{B_2}\;\right|\mathcal{F}_{S_1}\right]\Indic{B_1}\right]\Psi\Phi_n(x-y)\nonumber \quad \mbox{(by (\ref{eq:auxdensity1}))}\\
& = \mathds{E}_{\bm{\beta}}\left[\left(\bm{A}_{S_1}D^{-+}\dd y\right)\Indic{ B_1}\right]\Psi\Phi_n(x-y) \nonumber \\
&  = \bm{\beta}e^{C^-y}D^{-+}\Psi\Phi_n(x-y)\dd y \nonumber
\end{align}
%
by Theorem \ref{th:expandedB1}. 

Thus, (\ref{eq:Phistarn}) recursively holds for $n\ge 2$ with
\begin{align}\label{eq:Phinintegral}\Phi_{n}(x)=\int_0^x e^{C^-y}D^{-+}\Psi\Phi_{n-1}(x-y)\dd y,
\end{align}
which is unique by Lemma \ref{lem:bminimality}. Summing (\ref{eq:Phinintegral}) over $n\ge 0$ together with Fubini's Theorem, we obtain
\[
	\mathds{E}_{\bm{\beta}}\left[\bm{O}_x\Indic{\tau_-^x <\infty}\right]=\bm{\beta}\Phi(x)
	\]
for unique $\Phi(\cdot)$, which by Corollary \ref{cor:bminimality} is uniformly entrywise-bounded on compact intervals, and satisfies the equation
\[\Phi(x)=e^{C^-y} + \int_0^x e^{C^-y}D^{-+}\Psi\Phi(x-y)\dd y.\]
Theorem \ref{th:explicitphi} implies that $\Phi(x)=e^{(C^{-} +D^{-+}\Psi) x}$, which completes the proof.
\end{proof}

The following corollary shows how the mean value of $\bm{O}_x$ relates to the probability that the process $\{R_t\}_{t\ge 0}$ ever downcrosses level $-x\le 0$ given $\bm{A}_0\in\mathfrak{Z}^+$.

\begin{Corollary}\label{cor:ruin3}
For $\bm{\alpha}\in\mathfrak{Z}^+$,
\begin{align}
	\mathds{E}_{\bm{\alpha}}\left[\bm{O}_x\Indic{\tau_-^x <\infty}\right] & =\bm{\alpha}\Psi e^{(C^{-} +D^{-+}\Psi) x},  \\
	\label{eq:ruinprobx1}
\mathds{P}_{\bm{\alpha}}(\tau_-^x < \infty) & = \bm{\alpha}\Psi e^{(C^{-} +D^{-+}\Psi) x}\bm{1}.
\end{align}
\end{Corollary}
\begin{proof}
The strong Markov property, Corollary \ref{cor:algorithmpsi} and Theorem \ref{th:Gxbeta} imply that
\begin{align*}
\mathds{E}_{\bm{\alpha}} \left[\bm{O}_x\Indic{\tau_-^x <\infty}\right]  & =\mathds{E}_{\bm{\alpha}}\left[\mathds{E}_{\bm{\alpha}}\left[\left.\bm{O}_x\Indic{\tau_-^x <\infty}\;\right|\mathcal{F}_{\tau_-}\right]\Indic{\tau_-<\infty}\right]\\
&=\mathds{E}_{\bm{\alpha}}\left[\left(\bm{A}_{\tau_-}e^{(C^- +D^{-+}\Psi) x}\right)\Indic{\tau_-<\infty}\right] \\
& =\bm{\alpha}\Psi e^{(C^- +D^{-+}\Psi) x}.
\end{align*} 
Equation (\ref{eq:ruinprobx1}) follows because when $\tau_{-}^x <\infty$,  $\bm{O}_x\bm{1}=1$ due to the affine nature of $\mathfrak{Z}^-$.
\end{proof}

\subsection{Stationary distribution of the queue}
	\label{sec:stationaryz0}
In this section we determine the limiting behaviour of the process $\mathcal{Q} = \{Q_t\}_{t\ge 0}$ defined by
\[Q_t := R_t +\max\left\{0, \sup_{0\le s\le t} -R_s\right\}.\]
The process $\mathcal{Q}$ is the process $\mathcal{R}$ regulated at level $0$. We refer to $\{(Q_t, \bm{A}_t)\}_{t\ge 0}$ as the \emph{RAP-modulated fluid queue}, and assume the following stability condition. 
%
\begin{Condition}
	\label{cond:posrec} 
The process $\{Q_t\}_{t\ge 0}$ is positive recurrent under the measure $\mathds{P}_{\bm{\alpha}}$ for all $\bm{\alpha}\in\mathfrak{Z}_+\cup\mathfrak{Z}_-$.
\end{Condition}
Condition~\ref{cond:posrec} implies that for all $\bm{\alpha}\in\mathfrak{Z}_+\cup\mathfrak{Z}_-$, 
	\begin{align} 
		\mathds{E}_{\bm{\alpha}}\left[\tau_-\right] & <\infty, \nonumber\\
		\mathds{P}_{\bm{\alpha}}(\tau_-^x<\infty) & =1 \quad \mbox{ for $x\ge 0$,} \nonumber\\
		\Psi\bm{1} & =\bm{1}. \label{eq:psioneeqone}
	\end{align} 
%

In order to study the limiting behaviour of $\mathcal{Q}$, first let us compute
\[\lim_{t\rightarrow\infty} \mathds{E}\left[\bm{A}_t\Indic{Q_t=0,\;\bm{A}_t\in\mathfrak{Z}^-}\right].\]
Define the time-change
\[
	z_t:=\int_0^t\Indic{Q_s = 0, \;\bm{A}_s\in\mathfrak{Z}^-}\dd s,\qquad t\ge 0.
\]
A pathwise inspection reveals that, by Assumption~\ref{assump:pm1}, on the event $\{Q_t=0, \bm{A}_t\in\mathfrak{Z}^-\}$ we have $\bm{A}_t = \bm{O}_{z_t}$.
%
Thus,
\begin{align}
\lim_{t\rightarrow\infty} \mathds{E}\left[\bm{A}_t\Indic{Q_{t}=0,\; \bm{A}_t\in\mathfrak{Z}^-}\right] & = \lim_{t\rightarrow\infty} \mathds{E}\left[\bm{O}_{z_t}\Indic{Q_{t}=0,\; \bm{A}_t\in\mathfrak{Z}^-}\right]\nonumber\\
&=\lim_{y\rightarrow\infty}c_- \mathds{E}\left[\bm{O}_{y}\right],\label{eq:statdown1}
\end{align}
where $c_- :=\lim_{t\rightarrow \infty}\mathds{P}(Q_t=0, \bm{A}_t\in\mathfrak{Z}^-)$, the proportion of time $\{Q_t\}_{t\ge 0}$ spends in level $0$. We will compute the precise value of $c_-$ in (\ref{eq:c0condition1}), for now let it be unknown.
%

Since, by Condition~\ref{cond:posrec} $R_t\rightarrow -\infty$ as $t\rightarrow\infty$, this implies the first hitting time $\tau_{-}^y < \infty$ for all level $-y$, and therefore $\bm{O}_y\bm{1}=1$ almost surely. 
%
%
Thus, Theorem \ref{th:Gxbeta} implies that under Condition~\ref{cond:posrec} and in the case $\orbit$ starts in $\bm{\beta} \in \mathfrak{Z}^-$,
	\begin{align}\label{eq:Ox1eq1}
	\bm{\beta} e^{(C^- + D^{-+}\Psi) y}\bm{1} = \mathds{E}_{\bm{\beta}} \left[\bm{O}_{y}\bm{1}\right]=1 \quad \mbox{for all } y\ge 0.
	\end{align}
%
	
	%
	
Since $\bm{A}_0$ can be chosen to take any value $\bm{\beta}$ in $\mathfrak{Z}^-$, and $\mathfrak{Z}^-$ attains the minimality property, then by Lemma~\ref{lem:bminimality} and (\ref{eq:Ox1eq1}) the unique solution to $\bm{\beta} \bm{x}=1$ is $\bm{x}=e^{(C^- + D^{-+}\Psi) y}\bm{1}$. Since $\bm{x}=\bm{1}$ is also a solution, it follows that $e^{(C^- + D^{-+}\Psi) y}\bm{1} = \bm{1}$ for all $y \geq 0$. Differentiating and evaluating at $y=0$ gives $(C^- + D^{-+}\Psi)\bm{1} = \bm{0}$, which implies that $0$ is an eigenvalue of $C^- + D^{-+}\Psi$. The following is a stronger condition needed for our analysis.
	%

\begin{Condition}
	\label{cond:eigen1}
The eigenvalue~$0$ of the matrix $C^- + D^{-+}\Psi$ has multiplicity~$1$ and a normalised left eigenvector $\bm{v}_0$ (i.e., $\bm{v}_0 \bm{1} = 1)$.
\end{Condition}
%
%
%
%
Assume Condition \ref{cond:eigen1}. Then, by \cite[Lemma 4.1(b)]{asm:99},
$ e^{(C^- + D^{-+}\Psi) y} = \bm{1}\bm{v}_0 + o(e^{-\varepsilon y})$
for some $\varepsilon>0$.  This implies that in the case $\orbit$ starts in $\bm{\beta} \in \mathfrak{Z}^-$,
 \[
 	\lim_{y\rightarrow\infty}\mathds{E}_{\bm{\beta}} \left[\bm{O}_{y}\right] = \bm{\beta}(\bm{1}\bm{v}_0) = (\bm{\beta}\bm{1})\bm{v}_0 = \bm{v}_0,\]
Similarly, using Corollary \ref{cor:ruin3} and Eq. (\ref{eq:psioneeqone}) we get that in the case $\orbit$ starts in $\mathfrak{Z}^+$,
 \[	
 	\lim_{y\rightarrow\infty}\mathds{E}_{\bm{\beta}}\left[\bm{O}_{y}\right] = \bm{\beta}\Psi(\bm{1}\bm{v}_0) = \bm{\beta}(\Psi\bm{1})\bm{v}_0 = (\bm{\beta}\bm{1})\bm{v}_0 = \bm{v}_0.
\]
 %
 
Consequently, independently of the starting point of $\orbit$, by (\ref{eq:statdown1}),
\begin{align}	
	\label{eq:limitOy1}
	\lim_{t\rightarrow\infty} \mathds{E}\left[\bm{A}_t\Indic{Q_{t}=0,\; \bm{A}_t\in\mathfrak{Z}^-}\right]  = c_-\bm{v}_0.
	\end{align}
%

While (\ref{eq:limitOy1}) gives us a good indication of the expected behaviour of $\bm{A}_t$ for large values of $t$ with $Q_t=0$, we still need to analyse what happens to $\orbit$ between the epochs at which $\mathcal{Q}$ leaves the boundary $0$ and the epochs at which $\mathcal{Q}$ returns to $0$. Thus, we now focus on the properties of $\orbit$ while on a excursion of $\mathcal{Q}$ away from $0$. In the following, we study the process $\mathcal{R}$ up to time $\tau_-$, which is identical to the process $\mathcal{Q}$ up to time $\tau_-$. 
For 
$x\ge 0$, let
\begin{align*}
%
\mathcal{U}^x &:=\{u\in [0,\tau_-): R_u=x, \boldsymbol{A}_u\in\mathfrak{Z}^+\}, \quad \mathcal{D}^x :=\{u\in (0,\tau_-]: R_u=x, \boldsymbol{A}_u\in\mathfrak{Z}^-\}.
\end{align*}
Then $\mathcal{U}^x$ corresponds to the set of time epochs, $\{u_i^x\}_{i \geq 1}$, at which the process $\mathcal{R}$ upcrosses level~$x$ before $\tau_-$, while $\mathcal{D}^x$ corresponds to the set of time epochs, $\{d_i^x\}_{i \geq 1}$, at which $\mathcal{R}$ downcrosses level $x$ before~$\tau_-$. We compute the expected value of the sum of $\boldsymbol{\mathcal{A}}$ evaluated at each point in $\mathcal{U}^x$ and in $\mathcal{D}^x$ in Theorem~\ref{th:upcrossings3} and Corollary~\ref{cor:sumdx1}, respectively. 

\begin{theorem}
	\label{th:upcrossings3}
Let $x \ge 0$ and $\bm{\alpha}\in\mathfrak{Z}^+$. Then,
\begin{align}
\mathds{E}_{\bm{\alpha}}\left[\sum_{u_i^x \in \mathcal{U}^x}\bm{A}_{u_i^x}\right]= \bm{\alpha}e^{(C^+ + \Psi D^{-+})x}.
\end{align}
\end{theorem}

\begin{proof}

First, we classify the set of points $\{u_i^x\}_{i \geq 1}$ in $\mathcal{U}^x$ according to their \emph{complexity}. Define
\[
	\mathcal{U}^x_1 :=\left\{\begin{array}{cl}\{ u_1^x \}& \mbox{if }\bm{A}_s \in\mathfrak{Z}^+\;\forall s\in[0,x], \\
\vspace*{-0.3cm} \\
\varnothing &\mbox{otherwise}.\end{array}\right.
\] 
By (\ref{eq:leveltime1}), the set $\mathcal{U}^x_1$ contains the \emph{simplest} epoch at which $\mathcal{R}$ upcrosses $x$, in the sense that it upcrosses $x$ before any change of directions; in that case,  Assumption~\ref{assump:pm1} implies $u_1^x = x$.



Next, for $\mathcal{X} \in\Omega\setminus\Omega_1$  and for each $u^x_i \in \mathcal{U}^x \setminus \mathcal{U}^x_{1}$ define
\begin{align*} 
\gamma^x_i & := 
\inf\left\{y\in (0,x): \mbox{$\exists t\in[0,u^x_i)$ s.t. $R_t=y$, $\bm{A}_{t^-}\in\mathfrak{Z}^-$, $\bm{A}_{t}\in\mathfrak{Z}^+$}\right\},  \\
%
\xi^x_i & := \arg \inf\left\{y\in (0,x): \mbox{$\exists t\in[0,u^x_i)$ s.t. $R_t=y$, $\bm{A}_{t^-}\in\mathfrak{Z}^-$, $\bm{A}_{t}\in\mathfrak{Z}^+$}\right\}; 
\end{align*}  
$\gamma^x_i$ corresponds to the lowest level at which $\mathcal{R}$ has a down-up peak before time~$u^x_i$, and $\xi_i^x$ is the point in time of this down-up peak. 
%
%
For $n\ge 1$ recursively define
\begin{align*} 
	\mathcal{U}^x_{n+1} :=\left\{u^x_i \in \mathcal{U}^x \setminus \mathcal{U}^x_{1}: u_i^x - \xi_i^x \in \mathcal{U}^{x-\gamma_i^x}_{n}\circ\theta_{\xi_i^x}\right\},
\end{align*} 
where $\mathcal{U}^{x - \gamma_i^x}_{n}\circ\theta_{\xi_i^x}$ is the set $U_n^{x - \gamma_i^x}$ of upcrossing times to level $x - \gamma_i^x$ of the $\theta_{\xi_i^x}$-shifted path.
%
%
Thus, each epoch $u^x_i \in \mathcal{U}^x_{n+1}$ is such that $u_i^x - \xi_i^x$ is an upcrossing epoch in the set $\mathcal{U}^{x - \gamma_i^x}_{n}$of the $\theta_{\xi_i^x}$-shifted path; the set $\mathcal{U}^{x - \gamma_i^x}_{n}$ is of a \emph{lower complexity}\footnote{Another way to view this partition of sets: $u\in\mathcal{U}^x_n$ if and only if the time-reversed process $\{R_{u-s}-R_u\}_{s=0}^u$ reaches level $-x$ (at time $u$) with $V_x=n$.}. 
This implies that the collection $\{\mathcal{U}^x_n\}_{n\ge 1}$ is a partition of $\mathcal{U}^x$. See Figure \ref{fig:stationary1} for an illustration.
%
%
\begin{figure}[h]
\centering 
	\includegraphics[scale=0.44]{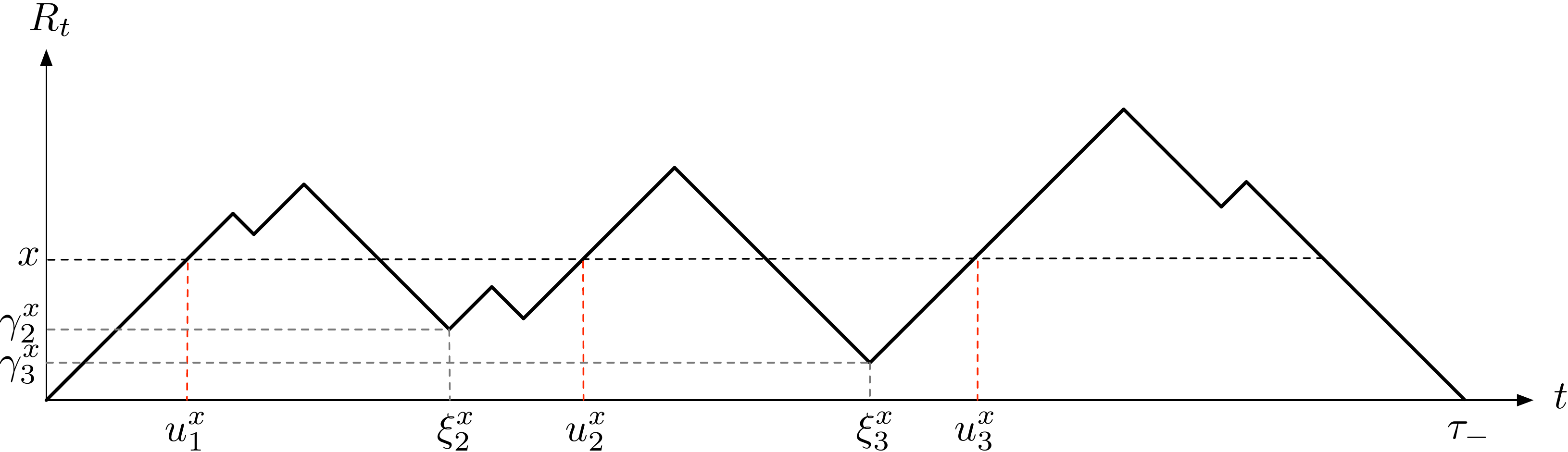}
\caption{Upcrossing times $u_1^x, u_2^x, u_3^x \in\mathcal{U}^x$ before $\tau_-$. Note that $u_1^x \in\mathcal{U}^x_1$, $u_2^x \in\mathcal{U}^x_3$, $u_3^x \in\mathcal{U}^x_2$.}
\label{fig:stationary1}
\end{figure}

First, we claim by induction that for all $n\ge 1$ and $x > 0$, 
\begin{align}
	\label{eq:upsilon1}
\mathds{E}_{\bm{\alpha}}\left[\sum_{u^x_i\in \mathcal{U}^x_n}\bm{A}_{u^x_i}\right]= \bm{\alpha}\Upsilon_n(x) ,
\end{align}
where $\{\Upsilon_n(\cdot)\}_{n\ge 1}$ are some continuous and unique matrices. 

{\bf Case $n=1$.}
  By (\ref{eq:leveltime1}), we have that
  \begin{align*}
  \mathds{E}_{\bm{\alpha}}\left[\sum_{u^x_1\in \mathcal{U}^x_1}\bm{A}_{u_i^x} \right] & = \mathds{E}_{\bm{\alpha}}\left[\bm{A}_x\Indic{\bm{A}_s\in\mathfrak{Z}^+\;\forall s\in[0,x]}\right] = \bm{\alpha}e^{C^+ x},
  \end{align*}
  so that (\ref{eq:upsilon1}) follows with $\Upsilon_1(x) = e^{C^+ x}$.
  
 {\bf Inductive part.} Suppose (\ref{eq:upsilon1}) holds for some $n\ge 1$. Let $y\in(0, x)$. Note that if $u_i^x, u_j^x \in \mathcal{U}^x_{n+1}$ and $\gamma_i^x,\gamma_j^x\in (y-\dd y, y)$, then 
 $\gamma_i^x=\gamma_j^x$ and $\xi_i^x= \xi_j^x$: this follows since each path in $\Omega$ there are no jumps that occur at the exact same level. 
 %
%
%
 This implies that 
\begin{align*}
\mathds{E}_{\bm{\alpha}} & \left[\sum_{u^x_i \in \mathcal{U}^x_{n+1}}\bm{A}_{u_i^x}\Indic{\gamma_i^x\in (y-\dd y, y)}\right] = \mathds{E}_{\bm{\alpha}}\left[\left[\sum_{u_i^{x - y} \in \mathcal{U}^{x-y}_{n}\circ\theta_{S_3}}\bm{A}_{u_i^{x - y}}\right]\Indic{B_1\cap B_2\cap B_{3}}\right],
\end{align*}
%
%
where
\[
	S_1 := y,\qquad S_2 := S_1 + \tau_-\circ\theta_{S_1},\qquad S_3 := S_2 + \tau_+\circ\theta_{S_2},
\]
and
\begin{align*}
B_1 & = \{\bm{A}_s\in\mathfrak{Z}^+\mbox{ for all }s \in [0,S_1)\},\\
B_2 & = \{\mathcal{X}\circ\theta_{S_1} \in\Omega_-\},\\
B_{3} & =\{\mbox{There exists $s\in (S_2, S_2 + \dd y)$ such that $\bm{A}_{s^-}\in\mathfrak{Z}^-$ and $\bm{A}_s\in\mathfrak{Z}^+$}\}. 
\end{align*}
%
Then,
\begin{align*}
& \mathds{E}_{\bm{\alpha}} \left[\sum_{u^x_i \in \mathcal{U}^x_{n+1}}\bm{A}_{u^x_i}\Indic{\gamma^x_i\in (y-\dd y, y)}\right] \\
 & =\mathds{E}_{\bm{\alpha}}\left[\mathds{E}_{\bm{\alpha}}\left[\left.\sum_{u_i^{x-y} \in \mathcal{U}^{x-y}_{n}\circ \theta_{S_3} }\bm{A}_{u^x_i}\;\right| \mathcal{F}_{S_3}\right]\Indic{B_1\cap B_2\cap B_{3}}\right] \\
&=\mathds{E}_{\bm{\alpha}}\left[\left(\bm{A}_{S_3}\Upsilon_n(x-y)\right)\Indic{B_1\cap B_2\cap B_{3}}\right] \quad \mbox{(by the induction hypothesis)} \\
& =  \mathds{E}_{\bm{\alpha}}\left[\mathds{E}_{\bm{\alpha}}\left[\left.\bm{A}_{S_3}\Indic{B_{3}}\;\right|\mathcal{F}_{S_2}\right]\Indic{B_1\cap B_2}\right]\Upsilon_n(x-y)\nonumber\\
& = \mathds{E}_{\bm{\alpha}}\left[\left(\bm{A}_{S_2}D^{-+}\dd y\right)\Indic{B_1\cap B_2}\right]\Upsilon_n(x-y) \quad \mbox{(by (\ref{eq:auxdensity1}))} \nonumber\\
& =  \mathds{E}_{\bm{\alpha}}\left[\mathds{E}_{\bm{\alpha}}\left[\left.\bm{A}_{S_2}\Indic{B_2}\;\right|\mathcal{F}_{S_1}\right]\Indic{B_1}\right]D^{-+}\Upsilon_n(x-y)\dd y\nonumber\\
& = \mathds{E}_{\bm{\alpha}}\left[\left(\bm{A}_{S_1}\Psi\right)\Indic{ B_1}\right]D^{-+}\Upsilon_n(x-y)\dd y \quad \mbox{(by Corollary \ref{cor:algorithmpsi})}\nonumber\\
& = \bm{\alpha}e^{C^+y}\Psi D^{-+}\Upsilon_n(x-y)\dd y \quad \mbox{(by Theorem \ref{th:expandedB1})}. 
\end{align*}
%

Thus, (\ref{eq:upsilon1}) recursively holds with
\begin{align}
	\label{eq:Upsilonnintegral}\Upsilon_{n}(x)=\int_0^x e^{C^+y}\Psi D^{-+}\Upsilon_{n-1}(x-y)\dd y.
\end{align}
Summing (\ref{eq:Upsilonnintegral}) over $n\ge 0$ and using the ergodicty of $\{Q_t\}_{t\ge 0}$ together with Fubini's Theorem, we obtain
\[\mathds{E}_{\bm{\alpha}}\left[\sum_{u_i^x \in \mathcal{U}^x}\bm{A}_{u^x_i}\right]=\bm{\beta}\Upsilon(x)\]
for a unique $\Upsilon(\cdot)$ which is uniformly entrywise-bounded on compact intervals and satisfies 
\[
	\Upsilon(x)=e^{C^+y} + \int_0^x e^{C^+y}\Psi D^{-+}\Upsilon(x-y)\dd y.
\]
Theorem \ref{th:explicitphi} implies that $\Upsilon(x)=e^{(C^+ +\Psi D^{-+}) x}$, which completes the proof.
\end{proof}

The expected value of the orbit at the points in $\mathcal{D}^x$ is computed as follows.

\begin{Corollary}
	\label{cor:sumdx1}
Let $x \ge 0$ and $\bm{\alpha}\in\mathfrak{Z}^+$. Then,
\begin{align}
\mathds{E}_{\bm{\alpha}}\left[\sum_{d_i^x \in \mathcal{D}^x}\bm{A}_{d_i^x} \right]= \bm{\alpha}e^{(C^+ + \Psi D^{-+})x}\Psi.
\end{align}
\end{Corollary}
\begin{proof}
Let $S=\tau_-\circ \theta_{u_i^x}$, the time elapsed from $u_i^x$ until the next downcrossing of level $x$. Then
\begin{align*}
\mathds{E}_{\bm{\alpha}}\left[\sum_{d_i^x \in \mathcal{D}^x}\bm{A}_{d_i^x}\right]& = \mathds{E}_{\bm{\alpha}}\left[\sum_{u_i^x \in \mathcal{U}^x}\bm{A}_{u_i^x +S }\right] \\
& = \mathds{E}_{\bm{\alpha}}\left[\sum_{u_i^x \in \mathcal{U}^x}\mathds{E}_{\bm{\alpha}}\left[\left.\bm{A}_{u_i^x+S }\;\right|\mathcal{F}_{u_i^x} \right]\right] \quad (\mbox{by Corollary \ref{cor:algorithmpsi}}) \\
& = \mathds{E}_{\bm{\alpha}}\left[\sum_{u_i^x \in \mathcal{U}^x}\bm{A}_{u_i^x} \Psi\right]  \\
& = \bm{\alpha}e^{(C^+ + \Psi D^{-+})x}\Psi,
\end{align*}
%
by Theorem \ref{th:upcrossings3}. 
\end{proof}

\begin{remark}
Since $\lim\limits_{t\rightarrow\infty}R_t=-\infty$ a.s., Theorem \ref{th:upcrossings3} implies that
	\begin{align*} 
		\lim_{x\rightarrow \infty}\bm{\alpha}e^{(C^+ + \Psi D^{-+})x}=\bm{0}\quad\mbox{for all } \bm{\alpha}\in\mathfrak{Z}^+.
	\end{align*} 
	Lemma \ref{lem:bminimality} implies that
	\begin{align*} 
		 \lim_{x\rightarrow \infty}e^{(C^+ + \Psi D^{-+})x}=0,
	\end{align*} 
	so that the dominant eigenvalue of $C^+ + \Psi D^{-+}$ has strictly negative real part.
\end{remark}

Now we are ready to state and prove the main result of this section regarding the limiting behaviour of $(\mathcal{Q}, \boldsymbol{\mathcal{A}})$.

%
\begin{theorem}
	\label{th:piplusminus1}
For $x>0$ define
\begin{align}
\Pi^+(x) & =\lim_{t\rightarrow\infty} \frac{\dd}{\dd x} \mathds{E}\left[\bm{A}_t\Indic{Q_t\in(0,x),\; \bm{A}_t\in\mathfrak{Z}^+}\right], \label{eq:Piplus1}\\
\Pi^-(x) & =\lim_{t\rightarrow\infty} \frac{\dd}{\dd x}\mathds{E}\left[\bm{A}_t\Indic{Q_t\in(0,x),\; \bm{A}_t\in\mathfrak{Z}^- }\label{eq:Piminus1}\right].
\end{align}
Then, 
\begin{align}
\Pi^+(x)& = c_-\bm{v}_0D^{-+}e^{(C^+ + \Psi D^{-+})x}\quad\mbox{ and}\label{eq:piplus}\\
\Pi^-(x) &= c_-\bm{v}_0D^{-+}e^{(C^+ + \Psi D^{-+})x}\Psi,\label{eq:piminus}
\end{align}
where $c_- :=\lim_{t\rightarrow \infty}\mathds{P}(Q_t=0, \bm{A}_t\in\mathfrak{Z}^-)$ and $\bm{v}_0$ is defined as in Condition \ref{cond:eigen1}.
\end{theorem}

\begin{proof} For $t > 0$, let $\chi_t :=\sup\{x > 0: Q_r>0\mbox{ for all }r\in(t- x, t]\}$; {see Figure \ref{fig:chit1} for a pathwise description of $\chi_t$.}

\begin{figure}[h]
	\begin{center} 
	\includegraphics[scale=0.44]{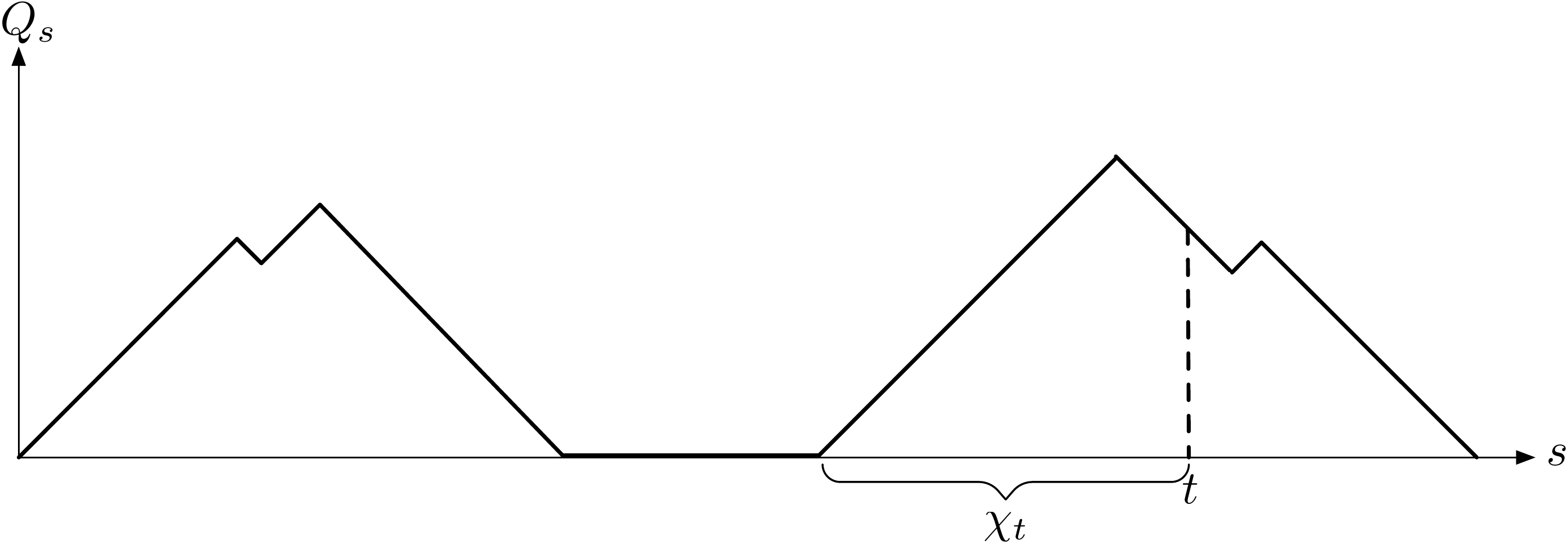} 
	\end{center} 
  \caption{The random variable $\chi_t$ associated to the RAP-modulated fluid queue $\{Q_s\}_{s\ge 0}$. Note that it corresponds to the elapsed time since $\{Q_s\}_{s\ge 0}$ left $0$ prior to $t$.}
  \label{fig:chit1}
\end{figure} 
 %
 %
Then, for $h> 0$,
\begin{align}
\lim_{t\rightarrow\infty} & \mathds{E}\left[\bm{A}_t\Indic{Q_t\in[x,x+h), \;\bm{A}_t\in\mathfrak{Z}^+ }\right]\nonumber\\
& = \lim_{t\rightarrow\infty} \int_{s=0}^{ t} \mathds{E}\left[\bm{A}_t\Indic{ Q_t\in(x,x+h),\; \bm{A}_t\in\mathfrak{Z}^+, \chi_t\in (s,s + \dd s) }\right]\nonumber\\
& = \lim_{t\rightarrow\infty} \int_{s=0}^{t} \mathds{E}\left[\mathds{E}\left[\left.\bm{A}_t\Indic{ Q_t\in[x,x+h),\; \bm{A}_t\in\mathfrak{Z}^+, \chi_t\in (s,s + \dd s) }\;\right| \mathcal{F}_{t-s}\right]\right]. \nonumber 
 \end{align} 
 Now, focusing on the inner expectation, we have 
 \begin{align} 
& \mathds{E}\left[\left.\bm{A}_t\Indic{ Q_t\in [x , x+h),\; \bm{A}_t\in\mathfrak{Z}^+, \chi_t\in (s,s + \dd s) }\;\right| \mathcal{F}_{t-s}\right] \nonumber \\
 & =  \mathds{E}\left[\left.\bm{A}_t\Indic{ Q_t\in [x , x+h),\; \bm{A}_t\in\mathfrak{Z}^+, Q_{t - s} = 0, \bm{A}_{t - s} \in \mathfrak{Z}^{-}, \bm{A}_{t - s \magenta{+} \dd s} \in \mathfrak{Z}^{+}} \;\right| \mathcal{F}_{t-s}\right] \nonumber \\
 & =  \mathds{E}_{\bm{A}_{t - s}}\left[\bm{A}_{s} \Indic{ Q_s \in [x , x+h),\; \bm{A}_s\in\mathfrak{Z}^+, \bm{A}_{0} \in \mathfrak{Z}^{-}, \bm{A}_{\dd s} \in \mathfrak{Z}^{+}} \right]\Indic{B_{t-s}}.   \label{eqn:mainaux4p5}
 \end{align} 
 where $B_t: = \{Q_{t}=0, \bm{A}_{t} \in \mathfrak{Z}^{-}\}$, $t\ge 0$. Thus, by Fubini's Theorem, we have for $h > 0$, 
\begin{align} 
	& \lim_{t\rightarrow\infty}  \mathds{E}\left[\bm{A}_t\Indic{Q_t\in[x,x+h), \;\bm{A}_t\in\mathfrak{Z}^+ }\right] \nonumber \\
	& =  \lim_{t\rightarrow\infty} \mathds{E}  \left[\int_{s=0}^{\infty}  \Indic{s\le t}\cdot\mathds{E}_{\bm{A}_{t - s}}\left[\bm{A}_{s} \Indic{ Q_s \in [x , x+h),\; \bm{A}_s\in\mathfrak{Z}^+, \bm{A}_{0} \in \mathfrak{Z}^{-}, \bm{A}_{\dd s} \in \mathfrak{Z}^{+}}\right]\Indic{B_{t-s}}\right]. \label{eqn:mainaux5} 
\end{align} 
In (\ref{eqn:mainaux5}), the Bounded Convergence Theorem allows us to switch the limit with the expectation and integral operators, change the variable $t-s$ to $t$ in the integrand (which is valid since both variables converge to $\infty$), and switch the limit with the expectation and integral operators once again. Thus,
\begin{align} 
	& \lim_{t\rightarrow\infty}  \mathds{E}\left[\bm{A}_t\Indic{Q_t\in[x,x+h), \;\bm{A}_t\in\mathfrak{Z}^+ }\right] \nonumber \\
	& =  \lim_{t\rightarrow\infty} \mathds{E}  \left[\int_{s=0}^{\infty}  \mathds{E}_{\bm{A}_{t }}\left[\bm{A}_{s} \Indic{ Q_s \in [x , x+h),\; \bm{A}_s\in\mathfrak{Z}^+, \bm{A}_{0} \in \mathfrak{Z}^{-}, \bm{A}_{\dd s} \in \mathfrak{Z}^{+}}\right]\Indic{B_{t}}\right]. \label{eqn:main} 
\end{align} 

Since the integrand in the expression above is computed on the event $\{\bm{A}_{t} \in\mathfrak{Z^-}\}$, w.l.o.g. suppose that $\bm{A}_{t} \in\mathfrak{Z}^-_i$ for some $i\in\mathcal{N}$. Then, 
\begin{align}
	 & \int_{s=0}^{\infty}  \mathds{E}_{\bm{A}_{t}}\left[\bm{A}_{s} \Indic{ Q_s \in [x , x+h),\; \bm{A}_s\in\mathfrak{Z}^+,  \bm{A}_{0} \in \mathfrak{Z}^{-}, \bm{A}_{\dd s} \in \mathfrak{Z}^{+}}\right]\Indic{B_t} \nonumber \\
	& =  \int_{s=0}^{\infty} \sum_{j \in \mathcal{N}} \mathds{E}_{\bm{\alpha}_{t}}\left[\bm{A}_{s} \Indic{ Q_s \in [x , x+h),\; \bm{A}_s\in\mathfrak{Z}^+}\right]\Indic{B_t} (\bm{A}_{t} \widehat{D}_{ij}^{-+}\bm{1}\dd s),  \nonumber 
\intertext{where $\bm{\alpha}_{t} := \frac{\bm{A}_{t} \widehat{D}^{-+}_{ij}}{\bm{A}_{t} \widehat{D}_{ij}^{-+}\bm{1}}$,}
	& = \sum_{j \in N} \mathds{E}_{\bm{\alpha}_{t}}\left[ \int_{s = 0}^{\infty} \bm{A}_{s} \Indic{ Q_s \in [x , x+h),\; \bm{A}_s\in\mathfrak{Z}^+} \dd s \right]\Indic{B_t} (\bm{A}_{t} \widehat{D}_{ij}^{-+}\bm{1}) \nonumber  \\
	& = \sum_{j \in N} \mathds{E}_{\bm{\alpha}_{t}}\left[ \sum_{u \in \mathcal{U}^x} \bm{A}_u h + o(h) \right]\Indic{B_t} (\bm{A}_{t} \widehat{D}_{ij}^{-+}\bm{1}) \label{eqn:integral}. 
\end{align} 
Substituting \eqref{eqn:integral} into \eqref{eqn:main} gives 
\begin{align} 
& \lim_{t\rightarrow\infty}  \mathds{E}\left[\bm{A}_t\Indic{Q_t\in[x,x+h), \;\bm{A}_t\in\mathfrak{Z}^+ }\right] \nonumber \\
%
%
& = \lim_{t \rightarrow \infty} \mathds{E} \left[\sum_{j \in N}\mathds{E}_{\bm{\alpha}_{t}}\left[ \sum_{u \in \mathcal{U}^x} \bm{A}_u h + o(h) \right]\Indic{B_t} (\bm{A}_{t} \widehat{D}_{ij}^{-+}\bm{1}) \right] \nonumber  \\
& = \lim_{t \rightarrow \infty}\mathds{E} \left[  \sum_{j \in N}\bm{\alpha}_{t} (e^{C^+ \Psi D^{-+}x}h)\Indic{B_t} (\bm{A}_{t} \widehat{D}_{ij}^{-+}\bm{1}) + o(h) \right] \nonumber \\ 
& = \lim_{t \rightarrow \infty}\mathds{E} \left[  \bm{A}_{t}D^{- +}  (e^{C^+ \Psi D^{-+}x}h)\Indic{B_t}  + o(h) \right] \nonumber \\ 
& = c_{-} \bm{v}_0D^{-+}e^{(C^+ + \Psi D^{-+})x}h + o(h), \nonumber 
\end{align} 
where (\ref{eq:limitOy1}) is used in the last equality. Equation (\ref{eq:piminus}) follows by analogous steps and arguments in the proof of Corollary \ref{cor:sumdx1}.
\end{proof}

To compute $c_-$, note that since $\bm{A}_t\bm{1}=1$ for all $t\ge 0$,
\begin{align}
1 & = \lim_{t\rightarrow\infty}\mathds{E}\left[\bm{A}_t\right]\bm{1}\nonumber\\
& = \left(\lim_{t\rightarrow\infty}\mathds{E}\left[\bm{A}_t\Indic{Q_t=0, \;\bm{A}_t\in\mathfrak{Z}^-}\right]\bm{1}\right) + \int_0^\infty \Pi^+(x)\bm{1}\dd x + \int_0^\infty \Pi^-(x)\bm{1}\dd x\nonumber\\
& = c_- + c_-\bm{v}_0\left[D^{-+}\int_0^\infty e^{(C^+ + \Psi D^{-+})x}\dd x\bm{1} + D^{-+}\int_0^\infty e^{(C^+ + \Psi D^{-+})x}\dd x(\Psi\bm{1})\right]\nonumber\\
& = c_-\left(1 - 2\bm{v}_0D^{-+}(C^+ + \Psi D^{-+})^{-1}\bm{1}\right),\quad\mbox{since }\Psi\bm{1}=\bm{1}.\label{eq:c0condition1}
\end{align}
By solving (\ref{eq:c0condition1}) for $c_-$ and using Theorem \ref{th:piplusminus1} we arrive at the following.
 \begin{Corollary}
 Let
 \begin{align}\label{eq:defpi1}\pi(x) = \lim_{t\rightarrow\infty}\frac{\dd }{\dd x }\mathds{P}(Q_t\in (0, x)),\qquad x\ge 0,\end{align}
which we call the \emph{stationary density} of $\{Q_t\}_{t\ge 0}$. Then,
\begin{align} 
 	\pi(x) & = 2c_-\bm{v}_0D^{-+}e^{(C^+ + \Psi D^{-+})x}\bm{1}, \\
	\lim_{t\rightarrow\infty}\mathds{P}(Q_t = 0, \bm{A}_t\in\mathfrak{Z}^-) & = c_-
\end{align} 
with $c_-=\left(1 - 2\bm{v}_0D^{-+}(C^+ + \Psi D^{-+})^{-1}\bm{1}\right)^{-1}$.
 \end{Corollary}

\section{Case $n^0>0$}
	\label{sec:withzerorates}

Now, we focus on the case 
that $\mathfrak{Z}^0$ is not empty and thus, $\mathcal{R}$ may have some piecewise-constant intervals. In this framework, the concepts of up-down and down-up peaks used throughout Section \ref{sec:nozerorates} are not sufficient. To address this issue, for $k\in\{+,-\}$ 
define
\[
	\rho^*_{k} :=\rho_{k} + \rho_0\circ\theta_{\rho_{k}},
\]
{where $\rho_\ell := \inf\{s\ge 0: \bm{A}_s\notin\mathfrak{Z}^\ell\}$ for $\ell\in\{+,-,0\}$. Then $\rho^*_k$ corresponds to 
the first exit time from $\mathfrak{Z}^0$ following an exit from $\mathfrak{Z}^k$.} Note that $\rho^*_{k}= \rho_k$ if and only if $\bm{A}_{\rho_k}\notin\mathfrak{Z}^0$.
%
\begin{Lemma}
	\label{lem:general1}
Let $k,\ell\in\{+,-\}$ and let $\bm{\alpha}\in\mathfrak{Z}^{k}$. Then, for all $x\ge 0$.
\begin{align*}
\mathds{E}_{\bm{\alpha}}&\left[\bm{A}_{\rho^*_{k}}\Indic{\rho_k\in(x,x+\dd x),\;\bm{A}_{\rho_k}\in\mathfrak{Z}^{0}, \;\bm{A}_{\rho^*_k}\in\mathfrak{Z}^{\ell}}\right] =\bm{\alpha}e^{C^{k}x}D^{k0}(-C^0)^{-1}D^{0\ell}\dd x\end{align*}
\end{Lemma}
\begin{proof}
Using the strong Markov property and Lemma \ref{lem:basicpropFRAP1}, we obtain
\begin{align*}
\mathds{E}_{\bm{\alpha}}&\left[\bm{A}_{\rho^*_k}\Indic{\rho_k\in(x,x+\dd x),\; \bm{A}_{\rho_k}\in\mathfrak{Z}^{0}, \;\bm{A}_{\rho^*_k}\in\mathfrak{Z}^{\ell}}\right]\\
& = \mathds{E}_{\bm{\alpha}}\left[\mathds{E}_{\bm{\alpha}}\left[\left.\bm{A}_{\rho^*_k}\Indic{\bm{A}_{\rho^*_k}\in\mathfrak{Z}^{\ell}}\;\right| \mathcal{F}_{\rho_k} \right]\Indic{\rho_k\in(x,x+\dd x),\; \bm{A}_{\rho_k}\in\mathfrak{Z}^0}\right]\\
& = \mathds{E}_{\bm{\alpha}}\left[\mathds{E}_{\bm{A}_{\rho_k}}\left[\bm{A}_{\rho_0}\Indic{\bm{A}_{\rho_0}\in\mathfrak{Z}^{\ell}}\right] \Indic{\rho_k\in(x,x+\dd x),\; \bm{A}_{\rho_k}\in\mathfrak{Z}^0}\right]\\
& = \mathds{E}_{\bm{\alpha}}\left[\left(\bm{A}_{\rho_k}(-C^0)^{-1}D^{0\ell}\right) \Indic{\rho_k\in(x,x+\dd x),\; \bm{A}_{\rho_k}\in\mathfrak{Z}^0}\right]\\
& = \mathds{E}_{\bm{\alpha}}\left[\bm{A}_{\rho_k} \Indic{\rho_k\in(x,x+\dd x), \;\bm{A}_{\rho_k}\in\mathfrak{Z}^0}\right](-C^0)^{-1}D^{0\ell}\\
& = \bm{\alpha}e^{C^{k} x}D^{k0}(-C^0)^{-1}D^{0\ell}\dd x.
\end{align*}
\end{proof}

Lemma \ref{lem:general1}  is analogous to \emph{censoring} in the context of stochastic fluid processes. In general terms, it concerns the inspection of certain aspects of a process before and after some interval; in the case of Lemma \ref{lem:general1} such an interval is $(\rho_k, \rho^*_k)$, during which $\bm{A}_{t} \in \mathfrak{Z}^0$, and the aspects inspected are $\rho_k$ and $\bm{A}_{\rho^*_k}$. 

When $n^0>0$, we say that an \emph{up-down peak} of $\mathcal{R}$ occurs at time $t>0$ if there exists some $s\in(0, t)$ such that $\rho^*_{+}\circ\theta_{s}=t-s$ and $\bm{A}_t\in\mathfrak{Z}_-$. That is, an up-down peak results from $\{\bm{A}_t\}_{t\ge 0}$ exiting $\mathfrak{Z}^+$ and either going directly to $\mathfrak{Z}^-$, or going through $\mathfrak{Z}^0$ and jumping later to $\mathfrak{Z}^-$. Analogously, a \emph{down-up peak} happens at time $t>0$ if there exists some $s\in(0, t)$ such that $\rho^*_{-}\circ\theta_{s}=t-s$ and $\bm{A}_t\in\mathfrak{Z}_+$. 

In the following we compute the expected value of the orbit at a peak for the general setting.

\begin{Corollary}\label{cor:general1}
Let $k,\ell\in\{+,-\}$, $k\neq \ell$, and $\bm{\alpha}\in\mathfrak{Z}^k$. Then,
\begin{align}
\label{eq:Dstaraux3}
	\mathds{E}_{\bm{\alpha}}\left[\bm{A}_{\rho^*_k}\Indic{\rho_k\in(x,x+\dd x), \;\bm{A}_{\rho^*_k}\in\mathfrak{Z}^{\ell}}\right]=\bm{\alpha}e^{C^{k} x}D^{k\ell*}\dd x,\end{align}
where $D^{k\ell*} = D^{k\ell} + D^{k 0}(-C^0)^{-1}D^{0\ell}.$ 
\end{Corollary}

\begin{proof}
Condition the indicator function on the LHS of (\ref{eq:Dstaraux3}) on the events $\{\bm{A}_{\rho_k}\in\mathfrak{Z}^{\ell}\}$ and $\{\bm{A}_{\rho_k}\in\mathfrak{Z}^{0}\}$. The result follows by Lemma \ref{lem:basicpropFRAP1} and Lemma \ref{lem:general1}.
\end{proof}

Now that we have a notion for \emph{peaks}, we develop next a way to handle the intervals between peaks in the case $n^0>0$. For instance, suppose that $\bm{A}_0\in\mathfrak{Z}_+$. A first step would be to compute the expected value of $\orbit$ at the instant $\mathcal{R}$ reaches some level $x\ge 0$ given that no peak has occured until that epoch, similar in spirit to Theorem \ref{th:expandedB1}. Note that up to the time before the first up-down peak, the orbit process can perform jumps from/to $\mathfrak{Z}^+$ and $\mathfrak{Z}^0$, so the analysis differs from that of Section \ref{sec:nozerorates}. However, it turns out that the solution to this problem has a structure similar to that of Theorem \ref{th:expandedB1} once we censor the occupation times in $\mathfrak{Z}^0$, as we will see next.

For $t \geq 0$, define
 \begin{align*}
  W_t & := \int_0^t \Indic{\bm{A}_s\in\mathfrak{Z}^+\cup\mathfrak{Z}^-}\dd s, \quad \zeta_x := \inf\left\{t\ge 0:W_t > x\right\}.  
 \end{align*}
Thus, $W_t$ corresponds to the \emph{occupation time of $\{\bm{A}_t\}_{t\ge 0}$ in $\mathfrak{Z}^+\cup\mathfrak{Z}^-$ up to time $t$}, while $\zeta_x$ corresponds to the \emph{necessary time to reach $x$ units of occupation time in $\mathfrak{Z}^+\cup\mathfrak{Z}^-$}. The process $\{W_t\}_{t\ge 0}$ is continuous nondecreasing and $\{\zeta_x\}_{x\ge 0}$ is c\`adl\`ag. {A process analogous to $\{W_t\}_{t\ge 0}$ has been used in the stochastic fluid model literature, it being regarded as the total amount of fluid that has flowed into or out of the system \cite{bean2005hitting}.}

By Assumption~\ref{assump:pm1}, for $t,x\ge 0$ and $k\in\{+,-\}$
\begin{align}
	\label{eq:leveltime2}
	\{\bm{A}_t\in\mathfrak{Z}^k\cup\mathfrak{Z}^0\}\cap\left\{\left|R_{t+\zeta_x\circ\theta_t}-R_t\right|=x\right\}=\left\{\bm{A}_u\in\mathfrak{Z}^k\cup\mathfrak{Z}^0\;\forall u\in\left[t,t+\zeta_x^k\circ\theta_t\right)\right\},\end{align}
a relation similar to (\ref{eq:leveltime1}). In the following we investigate more about the event (\ref{eq:leveltime2}). 
\begin{theorem}
	\label{th:meanmixedUandZ}
Let $\bm{\alpha}\in\mathfrak{Z}^k$, $k\in\{+,-\}$ and $x\ge 0$. Then
\begin{align}
	\label{eq:meanUandZ} 
	\mathds{E}_{\bm{\alpha}}\left[\bm{A}_{\zeta_x}\Indic{\bm{A}_s\in\mathfrak{Z}^k\cup\mathfrak{Z}^0\;\forall s\in [0,\zeta_x]}\right]= \bm{\alpha}e^{C^{k*}x},
\end{align}
where $C^{k*}  = C^k + D^{k0}(-C^0)^{-1}D^{0k}.$
\end{theorem}
\begin{proof}
Let $s_0^*:=\inf\{y>0: \bm{A}_{\zeta_y}\in\mathfrak{Z}^0\}=\inf\{t>0: \bm{A}_t\in\mathfrak{Z}^0\}$, and for $x\ge 0$ let
\[
	L_x^*:=\#\{s\in (0, \zeta_x]:\bm{A}_{s^-}\in\mathfrak{Z}^k, \bm{A}_s\in\mathfrak{Z}^0\}.
\]
We claim that there exist continuous matrices $\{\Sigma^*_n(\cdot)\}_{n\ge 0}$ such that for all $n\ge 0$,
\begin{align}
	\label{eq:Sigmanstar1}
\mathds{E}_{\bm{\alpha}}\left[\bm{A}_{{\zeta_x}}\Indic{\bm{A}_s\in\mathfrak{Z}^k\cup\mathfrak{Z}^0\;\forall s\in[0,\zeta_x],\; L_x^*=n}\right]=\bm{\alpha}\Sigma_n^*(x).
\end{align}
To prove this we use induction, following analogous steps to {those} in the proof of Theorem~\ref{th:expandedB1}.

{\bf Case $n=0$.} As ${\zeta_x}=x$ on $\{L^*_x=0\}$, we have 
\begin{align*}
\mathds{E}_{\bm{\alpha}}\left[\bm{A}_{{\zeta_x}}\Indic{\bm{A}_s\in\mathfrak{Z}^k\cup\mathfrak{Z}^0\;\forall s\in[0,\zeta_x],\; L_x^*=0}\right] = \mathds{E}_{\bm{\alpha}}\left[\bm{A}_{{\zeta_x}}\Indic{\bm{A}_s\in\mathfrak{Z}^k\;\forall s\in[0,x]}\right] = \bm{\alpha} e^{C^k x},
 \end{align*}
 where Theorem \ref{th:expandedB1} is used in the last equality. Thus, (\ref{eq:Sigmanstar1}) follows by choosing $\Sigma_0^*(x)= e^{C^k x}$, and this solution is unique by Lemma \ref{lem:bminimality}.

{\bf Inductive part.}
Suppose that (\ref{eq:Sigmanstar1}) holds for some $n\ge 0$. Then, for $r\in(0,x)$
\begin{align*}
\mathds{E}_{\bm{\alpha}}&\left[\bm{A}_{{\zeta_x}}\Indic{\bm{A}_s\in\mathfrak{Z}^k\cup\mathfrak{Z}^0\;\forall s\in[0,\zeta_x],\; L_x^*=n+1, s_0^*\in (r- \dd r, r)}\right]\\
& = \mathds{E}_{\bm{\alpha}}\left[\mathds{E}_{\bm{\alpha}}\left[\left.\bm{A}_{\zeta_x}\Indic{\bm{A}_s\in\mathfrak{Z}^k\cup\mathfrak{Z}^0\;\forall s\in[0,\zeta_x],\; L_x^*=n+1}\;\right|\mathcal{F}_{\zeta_r}\right]\Indic{s_0^*\in (r- \dd r, r),\; \bm{A}_{\zeta_r}\in\mathfrak{Z}^k}\right]\\
& = \mathds{E}_{\bm{\alpha}}\left[\mathds{E}_{\bm{A}_{\zeta_r}}\left[\bm{A}_{\zeta_{x-r}}\Indic{\bm{A}_s\in\mathfrak{Z}^k\cup\mathfrak{Z}^0\;\forall s\in[0,\zeta_{x-r}],\; L_{x-r}^*=n}\right]\Indic{s_0^*\in (r- \dd r, r),\; \bm{A}_{\zeta_r}\in\mathfrak{Z}^k}\right]\\
& = \mathds{E}_{\bm{\alpha}}\left[\left(\bm{A}_{\zeta_r}\Sigma^*_n(x-r)\right)\Indic{s_0^*\in (r- \dd r, r), \;\bm{A}_{\zeta_r}\in\mathfrak{Z}^k}\right]\\
& = \mathds{E}_{\bm{\alpha}}\left[\bm{A}_{\zeta_{s_0}}\Indic{s_0^*\in (r- \dd r, r), \;\bm{A}_{\zeta_{s_0}}\in\mathfrak{Z}^k}\right]\Sigma^*_n(x-r)\\
& = \bm{\alpha}e^{C^k r}D^{k0} (-C^0)^{-1}D^{0k}\Sigma^*_n(x-r)\dd r,
\end{align*}
where 
Lemma \ref{lem:general1} was used in the last equality.

Thus, (\ref{eq:Sigmanstar1}) recursively holds for $n\ge 0$ as
\begin{align}\label{eq:Sigmastarnintegral}\Sigma_{n+1}^*(x)=\int_0^x e^{C^k r}D^{k0} (-C^0)^{-1}D^{0k}\Sigma_{n}^*(x-r)\dd r,
\end{align}
which is unique by Lemma \ref{lem:bminimality}. Summing (\ref{eq:Sigmastarnintegral}) over $n\ge 0$, using Fubini's Theorem and Theorem \ref{th:explicitphi} we get that
\[\mathds{E}_{\bm{\alpha}}\left[\bm{A}_{\zeta_x}\Indic{\bm{A}_s\in\mathfrak{Z}^k\cup\mathfrak{Z}^0\;\forall s\in[0,\zeta_x]}\right]= \bm{\alpha}\Sigma^*(x),\]
where $\Sigma^*(x)=e^{C^{k*}x}$.
\end{proof}

Using Corollary \ref{cor:general1} and Theorem \ref{th:meanmixedUandZ} and the generalized concept of peaks we can develop the theory of first passages for general RAP-modulated fluid processes in virtually the same way as in Sections \ref{sec:firstzempty} and \ref{sec:minimum} and part of Section \ref{sec:stationaryz0}. We present the final formulae below.

\begin{theorem}
	\label{th:downcrossingmulti1}
Let $\mathcal{X}=(\mathcal{R},\bm{\mathcal{A}})$ be a general RAP-modulated fluid process with $n^0>0$. Let $\bm{\alpha}\in\mathfrak{Z}^+$.
Then 
\[\mathds{E}_{\bm{\alpha}}\left[\bm{A}_{\tau_-}\Indic{\Omega_\tau}\right] = \bm{\alpha}\Psi^*\]
where $\Psi^*=\lim_{n\rightarrow\infty}\Psi_n^*$ and $\{\Psi_n^*\}_{n\ge 0}$ are recursively computed by setting $\Psi_0^*=0$ and solving
\begin{align}
	\label{eq:FP3algmulti}
	C^{+*}\Psi_{n+1}^* + \Psi_{n+1}^*C^{-*}= - D^{+-*}-\Psi_{n}^*D^{-+*}\Psi_{n}^*.
\end{align}
Furthermore, for $x\ge 0$
\begin{align*}
  \mathds{E}_{\bm{\alpha}}\left[\bm{O}_x\Indic{\tau_-^x <\infty}\right]&=\bm{\alpha}\Psi^*e^{(C^{-*} +D^{-+*}\Psi^*) x},\\
\mathds{E}_{\bm{\alpha}}\left[\sum_{u\in \mathcal{U}^x}\bm{A}_u\right]  & =\bm{\alpha}e^{(C^{+*} + \Psi^*D^{-+*})x},\\
\mathds{E}_{\bm{\alpha}}\left[\sum_{u\in \mathcal{D}^x}\bm{A}_u\right]  &= \bm{\alpha}e^{(C^{+*} + \Psi^*D^{-+*})x}\Psi^*.
\end{align*}
\end{theorem}

We can compute $\Pi^+(\cdot)$ and $\Pi^-(\cdot)$ as defined in (\ref{eq:Piplus1}) and (\ref{eq:Piminus1}) by using analogous arguments to the ones used in the proof of Theorem \ref{th:piplusminus1}, however, there is a considerable difference. In (\ref{eqn:mainaux4p5}) we used that the event $\{Q_t=0, \bm{A}_t\in\mathfrak{Z}^+\}$ implies $\{\bm{A}_{t^-}\in\mathfrak{Z}^-\}$, however, this is no longer true in the case $n^0>0$. In fact, $\{Q_t=0, \bm{A}_t\in\mathfrak{Z}^+\}$ implies that $\bm{A}_{t^-}\in\mathfrak{Z}^-\cup\mathfrak{Z}^0$, so that a distinction needs to be made between occupations of $\{Q_t\}_{t\ge 0}$ in $(0,\infty)$ due to jumps coming from $\mathfrak{Z}^-$ or from $\mathfrak{Z}^0$. The first case was already addressed in the proof of Theorem \ref{th:piplusminus1}, while the second one follows by analogous arguments and by noticing that 
\begin{align*}
\lim_{t\rightarrow\infty} &\mathds{E}_{\bm{\beta}}\left[\bm{A}_t\Indic{Q_{t}=0,\; \bm{A}_t\in\mathfrak{Z}^0}\right]\\
&=\lim_{t\rightarrow\infty} \mathds{E}_{\bm{\beta}}\left[\int_{s=0}^\infty \mathds{E}_{\bm{A}_t}\left[\bm{A}_s\Indic{\rho_-\in (0, \dd s), \;\bm{A}_{r}\in\mathfrak{Z}^0\;\forall r\in[\rho_-,s]}\right]\Indic{Q_{t}=0, \;\bm{A}_t\in\mathfrak{Z}^-}\right]\\
&=\lim_{t\rightarrow\infty} \mathds{E}_{\bm{\beta}}\left[\left(\int_{s=0}^\infty \bm{A}_tD^{-0} e^{C^{0}s}\dd s\right)\Indic{Q_{t}=0, \;\bm{A}_t\in\mathfrak{Z}^-}\right]\\
&=\lim_{t\rightarrow\infty} \mathds{E}_{\bm{\beta}}\left[\bm{A}_t\Indic{Q_{t}=0, \;\bm{A}_t\in\mathfrak{Z}^-}\right]D^{-0} (-C^{0})^{-1}.
\end{align*}
Thus, the average limit orbit values in $\mathfrak{Z}^0$ on the event $\{Q_t=0\}$ can be computed once we compute the average limit orbit values in $\mathfrak{Z}^-$ on the event $\{Q_t=0\}$ as $t\rightarrow\infty$. The latter is computed similarly to (\ref{eq:statdown1}): for any $\bm{\beta}\in\mathfrak{Z}^-$,
\begin{align*}
\lim_{t\rightarrow\infty} \mathds{E}_{\bm{\beta}}\left[\bm{A}_t\Indic{Q_{t}=0, \;\bm{A}_t\in\mathfrak{Z}^-}\right]
&=\lim_{y\rightarrow\infty}c_-^*\mathds{E}_{\bm{\beta}}\left[\bm{O}_{y}\right]= c_-^*\bm{v}_0^*,
\end{align*}
where $c_-^*:=\lim_{t\rightarrow \infty}\mathds{P}(Q_t=0, \bm{A}_t\in\mathfrak{Z}^-)$ and $\bm{v}_0^*$ is the left eigenvector with $\bm{v}_0^*\bm{1}=1$ associated to the eigenvalue $0$ of $C^{+*} + \Psi^*D^{-+*}$. Analogous to Condition \ref{cond:eigen1}, we assume that such an eigenvector $\bm{v}_0^*$ is unique.

We now consider one final component
in order to compute the stationary distribution of $\{Q_t\}_{t\ge 0}$, which is
\begin{align}\label{eq:Pizero1}\Pi^0(x) :=\lim_{t\rightarrow\infty} \frac{\dd}{\dd x} \mathds{E}_{\bm{\beta}}\left[\bm{A}_t\Indic{Q_t\in(0,x), \;\bm{A}_t\in\mathfrak{Z}^0 }\right], \quad x> 0.
\end{align}
This can be computed using the following result whose proof is similar to that of Theorem \ref{th:piplusminus1}.

%
\begin{theorem}
Let $\bm{\alpha}\in\mathfrak{Z}^+$ and $h>0$. Then
\begin{align*}
& \mathds{E}_{\bm{\alpha}} \left[\int_{s=0}^\infty \bm{A}_s\Indic{s<\tau_-, \;Q_s\in[x,x+h), \;\bm{A}_s\in\mathfrak{Z}^0}\dd s\right] \\
& =\bm{\alpha}e^{(C^{+*} + \Psi^*D^{-+*})x}[D^{+0} + \Psi^*D^{-0}](-C^{0})^{-1} h + o(h).
\end{align*}
\end{theorem}

The previous leads to the following characterization of the stationary behaviour of $\{Q_t\}_{t\ge 0}$.
\begin{theorem}
Let $\Pi^+$, $\Pi^-$ and $\Pi^0$ be defined as in (\ref{eq:Piplus1}), (\ref{eq:Piminus1}) and (\ref{eq:Pizero1}), respectively. Then, for $x>0$
\begin{align*}
\Pi^+(x)& = c_-^*\bm{v}_0^*D^{-+*}e^{(C^{+*} + \Psi^*D^{-+*})x}\\
\Pi^-(x) &= c_-^*\bm{v}_0^*D^{-+*}e^{(C^{+*} + \Psi^*D^{-+*})x}\Psi^*\\
\Pi^0(x) &= c_-^*\bm{v}_0^*D^{-+*}e^{(C^{+*} + \Psi^*D^{-+*})x}[D^{+0} + \Psi^*D^{-0}](-C^0)^{-1},
\end{align*}
where $c_-^*:=\lim_{t\rightarrow \infty}\mathds{P}(Q_t=0, \bm{A}_t\in\mathfrak{Z}^-)$ is of the form
\begin{align*}
c_-^*=\left(1 - \bm{v}_0^*D^{-+*}(C^{+*} + \Psi^*D^{-+*})^{-1}(2\bm{1} + [D^{+0} + \Psi^*D^{-0}](-C^0)^{-1}\bm{1})\right)^{-1}.
\end{align*}
Furthermore, the stationary density function of $\{Q_t\}_{t\ge 0}$ as defined in (\ref{eq:defpi1}) is given by
\[\pi(x) = \Pi^+(x)\bm{1} + \Pi^-(x)\bm{1} + \Pi^0(x)\bm{1} = 2\Pi^+(x)\bm{1} + \Pi^0(x)\bm{1}, \quad x>0.\]
\end{theorem}
\section*{Acknowledgements.}
The first, second and fourth authors are affiliated with Australian Research Council (ARC) Centre of Excellence for Mathematical and Statistical Frontiers (ACEMS). The second and fourth authors also acknowledge the support of the ARC DP180103106 grant. 
\bibliographystyle{abbrv}
\bibliography{oscar}
\end{document}